\documentclass[11pt]{article}
\usepackage{amssymb,amsmath,amsfonts,amsthm,rotating,setspace,graphicx,float}
\usepackage{enumerate}
\usepackage{cases}

\numberwithin{equation}{section}

\newcommand{\I}{{\mathbf I}}
\newcommand{\J}{{\mathbf J}}

\def\R{{\mathbb R} }

\theoremstyle{plain}
\newtheorem{thm}{Theorem}[section]
\newtheorem{lem}{Lemma}[section]

\newtheorem{cor}{Corollary}[section]

\newtheorem{theorem}{Theorem}

\theoremstyle{definition}

\newtheorem{rem}{Remark}
\newtheorem{rmk}[theorem]{Remark}

\begin{document}

\title{\vskip-0.2in Pointwise Bounds and Blow-up for Systems of Semilinear Elliptic
  Inequalities at an Isolated Singularity via Nonlinear Potential Estimates}

\author{Marius Ghergu\footnote{School of Mathematical Sciences,
    University College Dublin, Belfield, Dublin 4, Ireland; {\tt
      marius.ghergu@ucd.ie}}
\footnote{Supported in part by the IRC Ulysses Grant 41942}, 
Steven D.~Taliaferro\footnote{Mathematics Department, Texas A\&M
    University, College Station, TX 77843-3368; {\tt stalia@math.tamu.edu}} 
\footnote{Corresponding author, Phone 001-979-845-3261, Fax
  001-979-845-6028},
Igor E.~Verbitsky\footnote{Department of Mathematics, University of
  Missouri, Columbia, MO 65211; {\tt verbitskyi@missouri.edu}}
\footnote{Supported in part by NSF grant DMS-1161622}}

\date{}
\maketitle


\begin{abstract}
We study the behavior near the origin of $C^2$ positive solutions $u(x)$ and $v(x)$ of the system 
$$
\begin{aligned}
0\leq -\Delta u\leq f(v)\\
0\leq -\Delta v\leq g(u)
\end{aligned}
\quad \mbox{ in }B_1(0)\setminus\{0\}\subset \R^n, n\geq 2,
$$
where $f,g:(0,\infty)\to (0,\infty)$ are continuous functions. We
provide optimal conditions on $f$ and $g$ at $\infty$ such that
solutions of this system satisfy pointwise bounds near the
origin. In dimension $n=2$ we show that this property holds if $\log^+
f$ or $\log^+g$ grow at most linearly at infinity. In dimension $n\geq
3$ and under the assumption $f(t)=O(t^\lambda)$, $g(t)=O(t^\sigma)$ as
$t\to \infty$, ($\lambda, \sigma\geq 0$), we obtain a new critical
curve that optimally describes the existence of such pointwise bounds.
Our approach relies in part on sharp estimates of nonlinear 
potentials which appear naturally in this context.
\medskip

\end{abstract}

\tableofcontents

\section{Introduction}\label{sec1}
In this paper we study the behaviour near the origin of $C^2$ positive
solutions $u(x)$ and $v(x)$ of the system
\begin{equation}\label{eq1.1}
\begin{aligned}
  0\leq -\Delta u\leq f(v)  \\  
 0\leq -\Delta v\leq g(u)
\end{aligned} \qquad \mbox{ in }B_1(0)\setminus\{0\}\subset \mathbb{R}^n, n\geq 2,
\end{equation}
where $f,g:(0,\infty)\to (0,\infty)$ are continuous functions.
More precisely, we consider the following question.
\medskip

\noindent {\bf Question 1}. For which continuous functions
$f,g:(0,\infty)\to(0,\infty)$ do there exist continuous functions $h_1
,h_2 :(0,1)\to(0,\infty)$ such that all $C^2$ positive solutions
$u(x)$ and $v(x)$ of the system \eqref{eq1.1}
satisfy
\begin{equation}\label{eq1.3}
 u(x)=O(h_1 (|x|)) \quad \text{ as } x\to 0
\end{equation}
\begin{equation}\label{eq1.4}
 v(x)=O(h_2 (|x|)) \quad \text{ as }x\to 0
\end{equation}
and what are the optimal such $h_1$ and $h_2$ when they exist?
\medskip

We call a function $h_1$ (resp. $h_2$) with the above properties a
pointwise bound for $u$ (resp. $v$) as $x\to0$.

Question 1 is motivated by the results on the single semilinear inequality  
$$
0\leq -\Delta u\leq f(u) \quad\mbox{ in }B_1(0)\setminus\{0\}\subset \mathbb{R}^n, n\geq 2,
$$
and its higher order version
$$
0\leq -\Delta^m u\leq f(u) \quad\mbox{ in }B_1(0)\setminus\{0\}\subset \mathbb{R}^n, m\geq 1, n\geq 2,
$$
which are discussed in \cite{T2001,T2006,T2013}.

Although the  literature on semilinear elliptic systems is
quite extensive, very little of it deals with semilinear
inequalities.  We mention the work of Bidaut-V\'eron and
Grillot \cite{BVG1999} in which the following coupled inequalities are
studied:
\begin{equation}\label{bvg1}
\begin{aligned}
  0\leq \Delta u\leq |x|^a v^p  \\  
 0\leq \Delta v\leq |x|^b u^q 
\end{aligned} \qquad \mbox{ in }B_1(0)\setminus\{0\}\subset
\mathbb{R}^n, \ n\geq 3,\ pq<1, 
\end{equation}
and
\begin{equation}\label{bvg2}
\begin{aligned}
  \Delta u\geq |x|^a v^p  \\  
 \Delta v\geq |x|^b u^q 
\end{aligned} \qquad \mbox{ in }B_1(0)\setminus\{0\}\subset
\mathbb{R}^n, \ n\geq 3,\ pq>1.
\end{equation}
 Another related system of semilinear elliptic inequalities appears in \cite{BVP2001} (see also \cite{MP2001}) and contains as a particular case the model
\begin{equation}\label{bvp}
\begin{aligned}
 - \Delta u\geq |x|^a v^p  \\  
 -\Delta v\geq |x|^b u^q 
\end{aligned} \qquad \mbox{ in } \mathbb{R}^n\setminus\{0\}, n\geq 2.
\end{equation}
Our system (1.1) is different in nature from \eqref{bvg1},
\eqref{bvg2} and \eqref{bvp} and its investigation completes the
general picture of semilinear elliptic systems of inequalities. In
particular (see Theorem 3.7), we will obtain pointwise bounds
for positive solutions of the system
 $$
\begin{aligned}
 0\leq - \Delta u\leq |x|^{a} v^p  \\  
 0\leq -\Delta v\leq |x|^{b} u^q 
\end{aligned} \qquad \mbox{ in } B_1(0)\setminus \{0\}, n\geq 3
$$
which complement the studies of systems \eqref{bvg1}, \eqref{bvg2} or
\eqref{bvp}.

\begin{rem}\label{rem1}
 Let
 \begin{equation}\label{eq1.5}
  \Gamma(r)=
  \begin{cases}
   r^{-(n-2)} , &\text{if }n\geq 3\\
   \log \frac{2}{r}, &\text{if }n=2.
  \end{cases}
 \end{equation}
Since $\Gamma(|x|)$ is positive and harmonic in $B_1
(0)\backslash \{0\}$, the functions $u_0 (x)=v_0
(x)=\Gamma(|x|)$ are always positive solutions of \eqref{eq1.1}.
Hence, any pointwise bound for positive solutions of \eqref{eq1.1} must
be at least as large as $\Gamma$ and whenever $\Gamma$ is such a
bound for $u$ (resp. $v$) it is necessarily optimal. In this case we
say that $u$ (resp. $v$) is {\it harmonically bounded} at $0$.
\end{rem}

We shall see that whenever a pointwise bound for positive solutions of
\eqref{eq1.1} exists, then $u$ or $v$ (or both) are harmonically bounded
at 0.

Our results reveal the fact that the optimal conditions for the
existence of pointwise bounds for positive solutions of \eqref{eq1.1}
are related to the growth at infinity of the nonlinearities $f$ and
$g$. In dimension $n=2$ we prove that pointwise bounds exist if
$\log^+f$ or $\log^+g$ grow at most linearly at infinity (see Theorems
\ref{thm2.1}, \ref{thm2.2} and \ref{thm2.3}). In dimensions
$n\geq 3$ we will assume that $f$ and $g$ have a power type growth at
infinity, namely
$$
f(t)=O(t^\lambda)\quad\mbox{ as }t\to\infty
$$
$$
g(t)=O(t^\sigma)\quad\mbox{ as }t\to\infty
$$
with $\lambda, \sigma\geq 0$. In this setting, we will find (see
Theorem \ref{thm3.4}) that no pointwise bounds exist if the pair
$(\lambda,\sigma)$ lies above the curve
\begin{equation}\label{curve}
\sigma=\frac{2}{n-2}+\frac{n}{n-2}\frac{1}{\lambda}.
\end{equation}
On the other hand, if $(\lambda,\sigma)$ lies below the curve
\eqref{curve} then pointwise bounds for positive solutions of
\eqref{eq1.1} always exist and their optimal estimates depend on new
subregions in the $\lambda\sigma$ plane (see Theorems \ref{thm3.1},
\ref{thm3.2}, \ref{thm3.3} and \ref{thm3.4}).

We note that the curve \eqref{curve} lies below the Sobolev hyperbola 
$$
\frac{1}{\sigma+1}+\frac{1}{\lambda+1}=1-\frac{2}{n}, 
\mbox{ that is, } \sigma=\frac{2\lambda+n+2}{(n-2)\lambda-2}
$$
which separates the regions of existence and nonexistence for
Lane-Emden systems:
$$
\begin{aligned}
  - \Delta u= v^\lambda  \\  
 -\Delta v=u^\sigma
\end{aligned} \qquad \mbox{ in } B_1(0)\subset\mathbb{R}^n, n\geq 3,
$$
(see \cite{M1993,M1996, PQS2007, S2009}).
\medskip

Our analysis of \eqref{eq1.1} combines the Brezis-Lions 
representation formula for superharmonic functions (see Appendix
\ref{secA}), a Moser type interation (see Lemma \ref{lem4.6}), and
certain pointwise estimates (see Corollary \ref{cor1})
for the nonlinear potential $N((Ng)^\sigma)$, $ \sigma \ge
\frac{2}{n-2}$, where $N$ is the Newtonian potential operator over a
ball in $\mathbb{R}^n$, $n\ge 3$, and $g$ is a nonnegative bounded
function.

Section \ref{nonlin} in this work is concerned with various pointwise and
integral estimates of nonlinear potentials of Havin-Maz'ya type and
their connections with Wolff potentials. 
Since the results in this section may be of independent interest, we
state them in greater generality than is needed for our study of the
system \eqref{eq1.1}.

In any dimension $n\geq 2$, we prove that our pointwise bounds for
positive solutions of \eqref{eq1.1} are optimal.  When these bounds
are not given by $\Gamma$, their optimality follows by constructing
(with the help of Lemma \ref{lem4.1}) solutions $u$ and $v$ of
\eqref{eq1.1} satisfying suitable coupled conditions on the union of a
countable number of balls which cluster at the origin and are harmonic
outside these balls.  In this case, it is interesting to point out
that although our optimal pointwise bounds are radially symmetric
functions, these bounds are not achieved by radial solutions of
\eqref{eq1.1}, because nonnegative radial superharmonic functions in a
punctured neighborhood of the origin are harmonically bounded as $x\to
0$.

We also consider the following analog of Question 1 when the
singularity is at $\infty$ instead of at the origin.
\medskip

\noindent {\bf Question 2}. For which continuous functions
$f,g:(0,\infty)\to(0,\infty)$ do there exist continuous functions $h_1
,h_2 :(1,\infty)\to(0,\infty)$ such that all $C^2$ positive solutions
$u(x)$ and $v(x)$ of the system
\begin{equation}\label{eq1.6}
 \begin{split} 
 0\leq -\Delta u\leq f(v)\\
 0\leq -\Delta v\leq g(u)
 \end{split}\quad
 \text{in }\mathbb{R}^n \backslash B_1 (0), \, n\geq 2,
 \end{equation}
 satisfy
\begin{equation}\label{eq1.8}
 u(x)=O(h_1 (|x|))\quad \text{ as }|x|\to\infty
\end{equation}
\begin{equation}\label{eq1.9}
 v(x)=O(h_2 (|x|)) \quad \text{ as }|x|\to\infty
\end{equation}
and what are the optimal such $h_1$ and $h_2$ when they exist?
\medskip

This paper is organized as follows. In Sections \ref{sec2} and
\ref{sec3} we state our main results in dimensions $n=2$ and $n\geq 3$
respectively.  In Section \ref{nonlin} we obtain, using Hedberg
inequalities and Wolff potential estimates, some new pointwise and
integral bounds for nonlinear potentials of Havin-Maz'ya type.  Using
these estimates, we collect in Section \ref{sec4} some preliminary
lemmas while Sections \ref{sec5} and \ref{sec6} contain the proofs of
our main results.

\section{Statement of two dimensional results}\label{sec2}

In this section we state our results for Questions 1 and 2 when $n=2$.

We say a continuous function $f:(0,\infty)\to(0,\infty)$ is
{\it exponentially bounded} at $\infty$ if
\begin{equation*}
 \log^{+} f(t)=O(t)\quad \text{ as }t\to\infty
\end{equation*}
where
\begin{equation*}
 \log^{+}s:=
 \begin{cases}
  \log s, &\text{if } s> 1\\
  0, &\text{if }s\leq 1.
 \end{cases}
\end{equation*}
If $f,g:(0,\infty)\to(0,\infty)$ are continuous functions then either
\begin{enumerate}
 \item $f$ and $g$ are both exponentially bounded at $\infty$;
 \item neither $f$ nor $g$ is exponentially bounded at $\infty$; or
 \item one and only one of the functions $f$ and $g$ is exponentially bounded at $\infty$.
\end{enumerate}
Our result for Question 1 when $n=2$ and $f$ and $g$ satisfy
(i) (resp. (ii), (iii)) is Theorem \ref{thm2.1} (resp. \ref{thm2.2},
\ref{thm2.3}) below.

By the following theorem, if the functions $f$ and $g$ are both
exponentially bounded at $\infty$ then all positive solutions $u$ and
$v$ of the system \eqref{eq1.1} are harmonically bounded at $0$.

\begin{thm}\label{thm2.1}
 Suppose $u(x)$ and $v(x)$ are $C^2$ positive solutions of the system
 \begin{equation}\label{eq2.1}
  0\leq -\Delta u\leq f(v)
 \end{equation}
 \begin{equation}\label{eq2.2}
   0\leq -\Delta v\leq g(u)
 \end{equation}
 in a punctured neighborhood of the origin in $\mathbb{R}^2$, where
 $f,g:(0,\infty)\to(0,\infty)$ are continuous and exponentially
 bounded at $\infty$.  Then both $u$ and $v$ are harmonically bounded,
 that is 
 \begin{equation}\label{eq2.3}
  u(x)=O\left( \log \frac{1}{|x|} \right) \quad \text{as }x\to0
 \end{equation}
 and
 \begin{equation}\label{eq2.4}
  v(x)=O\left(\log \frac{1}{|x|}\right) \quad \text{as }x\to0.
 \end{equation}
\end{thm}

By Remark \ref{rem1}, the bounds \eqref{eq2.3} and \eqref{eq2.4} are optimal.

By the following theorem, it is {\it essentially} the case that
if neither of the functions $f$ and $g$ is exponentially bounded at
$\infty$ then neither of the positive solutions $u$ and $v$ of the
system \eqref{eq1.1} satisfies an apriori pointwise bound at $0$.

\begin{thm}\label{thm2.2}
Suppose $f,g:(0,\infty)\to(0,\infty)$ are continuous functions satisfying
\begin{equation}\label{eq2.5}
 \lim_{t\to\infty} \frac{\log f(t)}{t}=\infty \quad\text{ and }\quad\lim_{t\to\infty} \frac{\log g(t)}{t}=\infty.
\end{equation}
Let $h:(0,1)\to(0,\infty)$ be a continuous function satisfying
$\lim_{r\to0^{+}} h(r)=\infty$.  Then there exist $C^2$ positive
solutions $u(x)$ and $v(x)$ of the system (\ref{eq2.1}, \ref{eq2.2}) in $B_1
(0)\backslash \{0\} \subset \mathbb{R}^2$ such that
\begin{equation}\label{eq2.6}
 u(x)\neq O(h(|x|)) \quad \text{as }x\to0
\end{equation}
and
\begin{equation}\label{eq2.7}
 v(x)\neq O(h(|x|)) \quad \text{as }x\to 0.
\end{equation}
\end{thm}

By the following theorem, if at least one of the functions $f$ and $g$
is exponentially bounded at $\infty$ then at least one of the positive
solutions $u$ and $v$ of the system \eqref{eq1.1} is harmonically bounded
at $0$.

\begin{thm}\label{thm2.3}
 Suppose $u(x)$ and $v(x)$ are $C^2$ positive solutions of the system
 \begin{align*}
  &0\leq -\Delta u\\
  &0\leq -\Delta v\leq g(u)
 \end{align*}
 in a punctured neighborhood of the origin in $\mathbb{R}^2$, where
 $g:(0,\infty)\to(0,\infty)$ is continuous and exponentially bounded
 at $\infty$.  Then $v$ is harmonically bounded, that is
 \begin{equation}\label{eq2.8}
  v(x)=O\left(\log \frac{1}{|x|}\right) \quad \text{as }x\to0.
 \end{equation}
 If, in addition,
 \begin{equation*}
  -\Delta u\leq f(v)
 \end{equation*}
 in a punctured neighborhood of the origin, where $f:(0,\infty)\to(0,\infty)$ is a continuous function satisfying 
 \begin{equation*}
  \log^{+} f(t)=O(t^\lambda ) \quad \text{as }t\to\infty
 \end{equation*}
 for some $\lambda>1$ then
 \begin{equation}\label{eq2.9}
  u(x)=o\left( \left(\log \frac{1}{|x|}\right)^{\lambda}\right) \quad \text{as }x\to 0.
 \end{equation}
\end{thm}

Note that in Theorems \ref{thm2.1}--\ref{thm2.3} we impose no
conditions on the growth of $f(t)$ (or $g(t)$) as $t\to0^{+}$.

By the following theorem, the bounds \eqref{eq2.9} and
\eqref{eq2.8} for $u$ and $v$ in Theorem \ref{thm2.3} are optimal.

\begin{thm}\label{thm2.4}
  Suppose $\lambda>1$ is a constant and $\psi:(0,1)\to(0,1)$ is a continuous
  function satisfying $\lim_{r\to 0^{+}} \psi (r)=0$.  Then there
  exist $C^\infty$ positive solutions $u(x)$ and $v(x)$ of the system
 \begin{equation}\label{eq2.10}
 \begin{aligned}
 0 & \le -\Delta u\le e^{v^\lambda}\\
 0 & \le -\Delta v\le e^u
 \end{aligned}
\qquad\text{in }B_2 (0)\backslash \{0\} \subset \mathbb{R}^2
 \end{equation}
 such that
 \begin{equation}\label{eq2.12}
  u(x)\neq O\left( \psi(|x|)\left(\log\frac{2}{|x|}\right)^{\lambda} \right)\quad \text{as }x\to0
 \end{equation}
 and
 \begin{equation}\label{eq2.13}
  \frac{v(x)}{\log\frac{1}{|x|}} \to 1 \quad \text{as }x\to0.
 \end{equation}
\end{thm}

The following theorem generalizes Theorems \ref{thm2.1} and
\ref{thm2.3} by allowing $u$ and $v$ to be negative and allowing the
right sides of (\ref{eq2.1}, \ref{eq2.2}) to depend on $x$.

\begin{thm}\label{thm2.5}
 Let $U(x)$ and $V(x)$ be $C^2$ solutions of the system
 \begin{equation}\label{eq2.14}
  0\leq -\Delta U\leq |x|^{-a} e^{|V|^\lambda}, \qquad U(x)>-a\log \frac{1}{|x|}
 \end{equation}
 \begin{equation}\label{eq2.15}
  0\leq -\Delta V\leq |x|^{-a} e^U, \qquad V(x)>-a \log\frac{1}{|x|}
 \end{equation}
 in a punctured neighborhood of the origin in $\mathbb{R}^2$ where $a$
 and $\lambda$ are positive constants.  Then
 \begin{equation}\label{eq2.16}
  U(x)=O\left( \log\frac{1}{|x|}\right)+o\left( \left(\log\frac{1}{|x|}\right)^\lambda \right) \quad \text{as }x\to0
 \end{equation}
 \begin{equation}\label{eq2.17}
  V(x)=O\left(\log\frac{1}{|x|}\right) \quad  \text{as }x\to0.
 \end{equation}
\end{thm}

The analog of Theorem \ref{thm2.5} when the singularity is at $\infty$
instead of at the origin is the following result.

\begin{thm}\label{thm2.6}
 Let $u(y)$ and $v(y)$ be $C^2$ solutions of the system
 \begin{equation*}
  0\leq -\Delta u\leq |y|^a e^{|v|^\lambda}, \qquad  u(y)>-a\log |y|
 \end{equation*}
 \begin{equation*}
  0\leq \Delta v\leq |y|^a e^u, \qquad v(y)>-a\log |y|
 \end{equation*}
 in the complement of a compact subset of $\mathbb{R}^2$ where $a$ and
 $\lambda$ are positive constants.  Then
\begin{equation} 
\begin{split}
  &u(y)=O(\log |y|)+o((\log |y|)^\lambda ) \\
  &v(y)=O(\log |y|) 
 \end{split} 
\quad \text{ as }|y|\to\infty.
\end{equation}
\end{thm}
 \begin{proof}
  Apply the Kelvin transform
  \begin{equation*}
   U(x)=u(y), \quad V(x)=v(y), \quad y=\frac{x}{|x|^2}
  \end{equation*}
  and then use Theorem \ref{thm2.5}.
 \end{proof}

\section{Statement of three and higher dimensional
  results}\label{sec3}

In this section we state our results for Questions 1 and 2 when
$n\geq 3$.  We will mainly be concerned with the case that the
continuous functions $f,g:(0,\infty)\to(0,\infty)$ in Questions 1
and 2 satisfy

\begin{equation}\label{eq3.1}
 f(t)=O(t^\lambda ) \quad \text{ as }t\to\infty
\end{equation}
\begin{equation}\label{eq3.2}
 g(t)=O(t^\sigma) \quad \text{ as }t\to\infty
\end{equation}
for some nonnegative constants $\lambda$ and $\sigma$.  We can assume
without loss of generality that $\sigma\leq \lambda$.

If $\lambda$ and $\sigma$ are nonnegative constants satisfying
$\sigma\leq \lambda$ then $(\lambda ,\sigma)$ belongs to one of the
following four pointwise disjoint subsets of the $\lambda
\sigma$-plane:
\begin{align*}
 &A:=\left\{ (\lambda ,\sigma): \, 0\leq \sigma \leq \lambda \leq \frac{n}{n-2}\right\} \\
 &B:=\left\{ (\lambda ,\sigma): \, \lambda>\frac{n}{n-2}
   \quad\text{and}
\quad 0\leq \sigma<\frac{2}{n-2}+\frac{n}{n-2} \frac{1}{\lambda} \right\} \\
 &C:=\left\{ (\lambda ,\sigma): \, \lambda>\frac{n}{n-2} \quad
   \text{and}\quad \frac{2}{n-2}+\frac{n}{n-2} \frac{1}{\lambda}<\sigma \leq \lambda \right\} \\
 &D:=\left\{ (\lambda ,\sigma): \, \lambda>\frac{n}{n-2} \quad
   \text{and}\quad \sigma=\frac{2}{n-2}+\frac{n}{n-2} \frac{1}{\lambda} \right\}.
\end{align*}
\begin{figure}[H]
 \includegraphics[scale=.65]{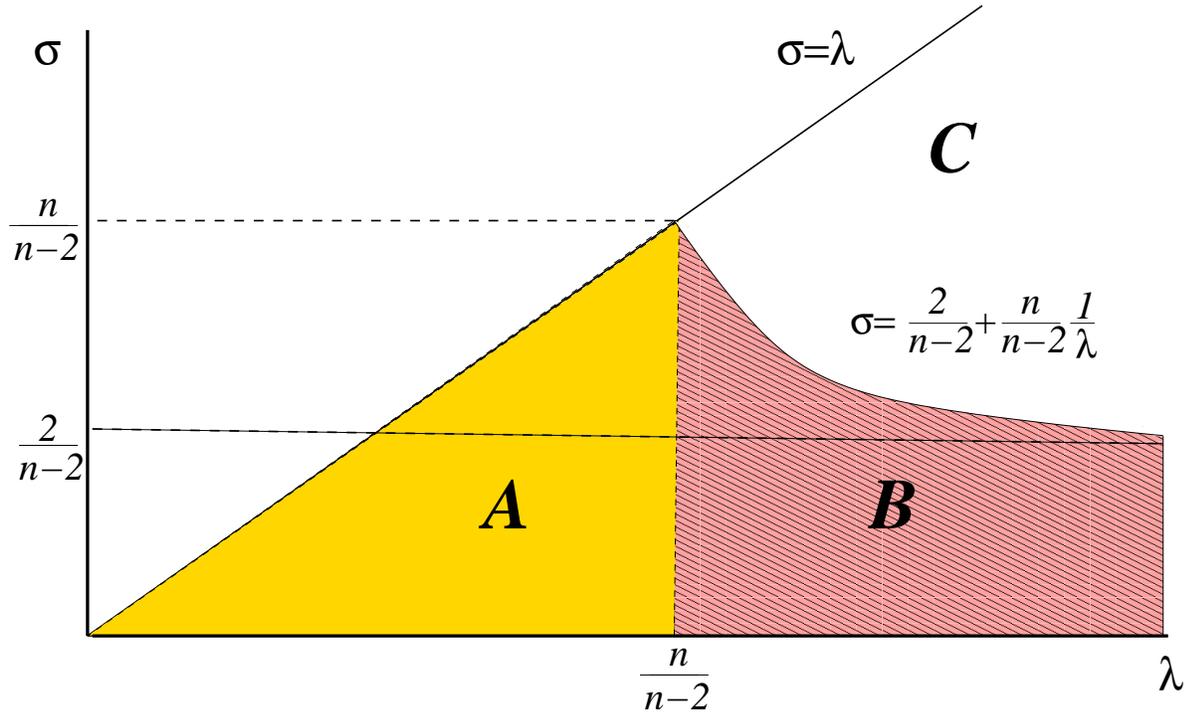}
 \caption{Graph of regions A, B and C.}
\end{figure}

Note that $A$, $B$ and $C$ are two dimensional regions in the $\lambda
\sigma$-plane whereas $D$ is the curve separating $B$ and $C$. (See
Figure 1.)

In this section we give a complete answer to Question 1 when
$n\geq 3$ and the functions $f$ and $g$ satisfy (\ref{eq3.1}, \ref{eq3.2}) where
$(\lambda ,\sigma)\in A\cup B\cup C$.  The following theorem deals
with the case that $(\lambda ,\sigma)\in A$.

\begin{thm}\label{thm3.1}
  Let $f,g:(0,\infty)\to(0,\infty)$ be continuous functions satisfying
  (\ref{eq3.1}, \ref{eq3.2}) where
 \begin{equation}\label{eq3.3}
  0\leq \sigma\leq \lambda\leq \frac{n}{n-2}.
 \end{equation}
 Suppose $u(x)$ and $v(x)$ are $C^2$ positive solutions of the system
 \begin{equation}\label{eq3.4}
  0\leq -\Delta u\leq f(v)
 \end{equation}
 \begin{equation}\label{eq3.5}
  0\leq -\Delta v\leq g(u)
 \end{equation}
 in a punctured neighborhood of the origin in $\mathbb{R}^n , \, n\geq
 3$.  Then both $u$ and $v$ are harmonically bounded, that is
 \begin{equation}\label{eq3.6}
  u(x)=O(|x|^{-(n-2)}) \quad \text{ as }x\to0
 \end{equation}
 \begin{equation}\label{eq3.7}
  v(x)=O(|x|^{-(n-2)}) \quad \text{ as }x\to0.
 \end{equation}
\end{thm}

By Remark \ref{rem1}, the bounds \eqref{eq3.6} and \eqref{eq3.7} are optimal.

The following two theorems deal with the case $(\lambda ,\sigma) \in B$.

\begin{thm}\label{thm3.2}
  Let $f,g:(0,\infty)\to(0,\infty)$ be continuous functions satisfying
  (\ref{eq3.1}, \ref{eq3.2}) where
 \begin{equation}\label{eq3.8}
  \lambda>\frac{n}{n-2} \quad\text{and}\quad
\sigma<\frac{2}{n-2}+\frac{n}{n-2} \frac{1}{\lambda}.
 \end{equation}
 Suppose $u(x)$ and $v(x)$ are $C^2$ positive solutions of the system
 (\ref{eq3.4}, \ref{eq3.5}) in a punctured neighborhood of the origin in $\mathbb{R}^n
 , \, n\geq 3$.  Then
 \begin{equation}\label{eq3.9}
  u(x)=o\left(|x|^{-\frac{(n-2)^2}{n}\lambda} \right) \quad \text{ as }x\to0
 \end{equation}
and
 \begin{equation}\label{eq3.10}
  v(x)=O\left(|x|^{-(n-2)} \right) \quad \text{ as }x\to0.
 \end{equation}
\end{thm}

By the following theorem the bounds \eqref{eq3.9} and \eqref{eq3.10} for $u$ and $v$ in Theorem \ref{thm3.2} are optimal.

\begin{thm}\label{thm3.3}
  Suppose $\lambda$ and $\sigma$ satisfy \eqref{eq3.8} and
  $\psi:(0,1)\to(0,1)$ is a continuous function satisfying $\lim_{r\to
    0^{+}} \psi(r)=0$.  Then there exist $C^\infty$ positive solutions
  $u(x)$ and $v(x)$ of the system
 \begin{equation}\label{eq3.11}
 \begin{aligned}
 0 & \le -\Delta u\le v^\lambda\\
 0 & \le -\Delta v\le u^\sigma
 \end{aligned}
 \text{\qquad in }\mathbb{R}^n \backslash \{0\}, \, n\geq 3
 \end{equation}
 such that
\begin{equation}\label{eq3.13}
  u(x)\neq O\left( \psi(|x|) |x|^{-\frac{(n-2)^2}{n}\lambda}\right) 
\quad \text{as }x\to0
 \end{equation}
 and
 \begin{equation}\label{eq3.14}
  v(x)|x|^{n-2} \to 1 \quad \text{as }x\to0.
 \end{equation}
\end{thm}

The following theorem deals with the case that $(\lambda ,\sigma)\in
C$.  In this case there exist pointwise bounds for neither $u$ nor
$v$.

\begin{thm}\label{thm3.4}
 Suppose $\lambda$ and $\sigma$ are positive constants satisfying
 \begin{equation}\label{eq3.15}
  \frac{2}{n-2}+\frac{n}{n-2} \frac{1}{\lambda}<\sigma\leq\lambda.
 \end{equation}
 Let $h:(0,1)\to(0,\infty)$ be a continuous function satisfying
 \begin{equation*}
  \lim_{r\to0^{+}} h(r)=\infty.
 \end{equation*}
 Then there exist $C^\infty$ solutions $u(x)$ and $v(x)$ of the system
 \begin{equation}\label{eq3.16}
 \left. 
 \begin{aligned}
 0 & \le -\Delta u\le v^\lambda\\
 0 & \le -\Delta v\le u^\sigma\\
 u & >1, \, v>1
 \end{aligned}
 \right\}
 \text{\qquad in } \mathbb{R}^n \setminus \{0\}, \, n\geq 3
 \end{equation}
 such that
 \begin{equation}\label{eq3.17}
  u(x)\neq O(h(|x|)) \quad \text{as }x\to0
 \end{equation}
 and
 \begin{equation}\label{eq3.18}
  v(x)\neq O(h(|x|))\quad \text{as }x\to0.
 \end{equation}
\end{thm}

The following theorem can be viewed as the limiting case of Theorem
\ref{thm3.2} as $\lambda\to\infty$.

\begin{thm}\label{thm3.5}
Let $g:(0,\infty)\to(0,\infty)$ be a continuous function satisfying
\eqref{eq3.2}  where
\[
\sigma<\frac{2}{n-2}.
\]
Suppose $u(x)$ and $v(x)$ are $C^2$ positive solutions of the system
\begin{align*}
0&\le -\Delta u\\
0&\le -\Delta v\le g(u)
\end{align*}  
in a punctured neighborhood of the origin in $\mathbb{R}^n$,
$n\ge3$. Then $v$ is harmonically bounded, that is
\begin{equation}\label{eq3.19}
v(x)=O\left(|x|^{-(n-2)}\right) \quad\text{as }x\to 0.
\end{equation}
\end{thm}

By Remark \ref{rem1}, the bound \eqref{eq3.19} is optimal.

In Theorem \ref{thm3.7} we will extend some of our results to the more
general system
\begin{align*}
 \notag &-\Delta u=|x|^{-\alpha} v^\lambda\\
 &-\Delta v=|x|^{-\beta}u^\sigma.
\end{align*}
Using these extended results and the Kelvin transform, we obtain the
following theorem concerning pointwise bounds for positive solutions
$U(y)$ and $V(y)$ of the system
\begin{equation}\label{eq3.20}
\begin{split}
 &0\leq -\Delta U\leq (V+1)^\lambda \\
 &0\leq -\Delta V\leq (U+1)^\sigma
\end{split}
\end{equation}
in the complement of a compact subset of $\mathbb{R}^n , \, n\geq 3$, where
\begin{equation}\label{eq3.21}
 \lambda\geq\sigma\geq0 \quad \text{ and }\quad\sigma<\frac{2}{n-2}+\frac{n}{n-2} \frac{1}{\lambda}.
\end{equation}
Note that $\lambda$ and $\sigma$ satisfy \eqref{eq3.21} if and only if
$(\lambda ,\sigma)\in\left( A\backslash \left\{
    \left(\frac{n}{n-2},\frac{n}{n-2}\right) \right\} \right) \cup B$
where $A$ and $B$ are defined at the beginning of this section and
graphed in Figure 1.

\begin{thm}\label{thm3.6}
  Let $U(y)$ and $V(y)$ be $C^2$ nonnegative solutions of the system
  \eqref{eq3.20} in the complement of a compact subset of
  $\mathbb{R}^n$, $n\geq 3$, where $\lambda$ and $\sigma$ satisfy
  \eqref{eq3.21}.
 \begin{description}
  \item[Case A.] If $\sigma=0$ then as $|y|\to\infty$
  \begin{align*}
   &U(y)=o\left(|y|^{\frac{n-2}{n}\left(\frac{2(n-2)}{n}\lambda+2\right)}\right)\\
   &V(y)=o\left(|y|^{\frac{2(n-2)}{n}}\right).
  \end{align*}
  \item[Case B.] If $0<\sigma<\frac{2}{n-2}$ then as $|y|\to\infty$
  \begin{align*}
   &U(y)=o\left(|y|^{\frac{2(n-2)(\lambda+1)}{n}} \right)\\
   &V(y)=O(|y|^2).
  \end{align*}  
  \item[Case C.] If $\sigma=\frac{2}{n-2}$ then as $|y|\to\infty$
  \begin{align*}
   &U(y)=o\left(|y|^{\frac{2(n-2)(\lambda+1)}{n}} (\log |y|)^{\frac{n-2}{n}\lambda} \right)\\
   &V(y)=o(|y|^2 \log |y|).
  \end{align*}
  \item[Case D.] Suppose $\sigma>\frac{2}{n-2}$.  Let $\varepsilon>0$ and $D=(n-2)\lambda \left(\frac{2}{n-2}+\frac{n}{n-2} \frac{1}{\lambda}-\sigma \right)$.  Then $D>0$ and as $|y|\to\infty$
  \begin{align*}
   &U(y)=o\left(|y|^{\frac{2(n-2)(\lambda+1)}{D}+\varepsilon}\right)\\
   &V(y)=o\left(|y|^{\frac{2(n-2)(\sigma+1)}{D}+\varepsilon} \right).
  \end{align*}
 \end{description}
\end{thm}

\begin{thm}\label{thm3.7}
 Let $u(x)$ and $v(x)$ be $C^2$ nonnegative solutions of the system
 \begin{equation}\label{eq3.22}
  0\leq -\Delta u\leq |x|^{-\alpha} \left(v+|x|^{-(n-2)}\right)^\lambda
 \end{equation}
 \begin{equation}\label{eq3.23}
  0\leq -\Delta v\leq |x|^{-\beta} \left(u+|x|^{-(n-2)}\right)^\sigma
 \end{equation}
 in a punctured neighborhood of the origin in $\mathbb{R}^n , \, n\geq
 3$, where $\alpha,\beta\in\mathbb{R}$ and $\lambda$ and $\sigma$
 satisfy \eqref{eq3.21}.
\begin{description}
 \item[Case A.] Suppose $\sigma=0$.
 \begin{description}
 \item[(A1)] If $\beta\leq n$ then as $x\to0$
 \begin{align*}
  &u(x)=O\left( \left(\frac{1}{|x|}\right)^{n-2} \right) + o\left( \left(\frac{1}{|x|}\right)^{\frac{n-2}{n}((n-2)\lambda+\alpha)} \right)\\
  &v(x)=O\left( \left(\frac{1}{|x|}\right)^{n-2} \right).
 \end{align*}
 \item[(A2)] If $\beta>n$ then as $x\to0$
 \begin{align*}
  &u(x)=O\left( \left(\frac{1}{|x|}\right)^{n-2} \right) +o\left( \left(\frac{1}{|x|}\right)^{\frac{n-2}{n}\left(\frac{n-2}{n}\beta\lambda+\alpha\right)} \right)\\
  &v(x)=o\left( \left(\frac{1}{|x|}\right)^{\frac{n-2}{n}\beta} \right).
 \end{align*}
 \end{description}
 \item[Case B.] Suppose $0<\sigma<\frac{2}{n-2}$.  Let
 \begin{equation}\label{eq3.24}
  \delta=\max \{(n-2)\lambda+\alpha, \, [(n-2)\sigma-2+\beta]\lambda+\alpha\}
 \end{equation}
  \begin{description}
  \item[(B1)] If $\delta\leq n$ then as $x\to0$
    \begin{align}\label{eq3.25}
   u(x)&=O\left( \left(\frac{1}{|x|}\right)^{n-2}\right)\\
  \label{eq3.26}
  v(x)&=
  \begin{cases}
  O\left( \left(\frac{1}{|x|}\right)^{n-2}\right), &\text{ if }\beta\leq n-(n-2)\sigma\\
  o\left( (\frac{1}{|x|})^{\frac{n-2}{n}[(n-2)\sigma+\beta]} \right), &\text{ if }\beta>n-(n-2)\sigma.
  \end{cases}   
  \end{align}
  \item[(B2)] If $\delta>n$ then as $x\to0$
  \begin{align}\label{eq3.27}
   u(x)&=o\left( \left(\frac{1}{|x|}\right)^{\frac{n-2}{n}\delta} \right)\\
  \label{eq3.28}
   v(x)&=O\left( \left(\frac{1}{|x|}\right)^{n-2} +\left(\frac{1}{|x|}\right)^{(n-2)\sigma-2+\beta} \right).
  \end{align}
  \end{description}
 \item[Case C.] Suppose $\sigma=\frac{2}{n-2}$.
  \begin{description}
  \item[(C1)] If either
  \begin{enumerate}
   \item $\beta<n-2$ and $(n-2)\lambda+\alpha\leq n; $ or
   \item $\beta\geq n-2$ and $\beta\lambda+\alpha<n$
  \end{enumerate}
   then, as $x\to0$, $u$ and $v$ satisfy \eqref{eq3.25} and
   \eqref{eq3.26}, that is
   \begin{align*}
    &u(x)=O \left( \left(\frac{1}{|x|}\right)^{n-2} \right)\\
    &v(x)=
    \begin{cases}
     O \left( \left(\frac{1}{|x|}\right)^{n-2} \right), &\text{ if }\beta\leq n-2\\
     o\left( \left(\frac{1}{|x|}\right)^{\frac{n-2}{n}(\beta+2)} \right), &\text{ if }\beta>n-2.
    \end{cases}
   \end{align*}
   \item[(C2)] If neither (i) nor (ii) holds then as $x\to0$
   \begin{align}\label{eq3.29}
    u(x)&=
    \begin{cases}
     o\left( \left(\frac{1}{|x|}\right)^{\frac{n-2}{n}[(n-2)\lambda+\alpha]}\right), &\text{ if }\beta<n-2\\
     o\left( \left(\frac{1}{|x|}\right)^{\frac{n-2}{n}(\beta\lambda+\alpha)} \left(\log \frac{1}{|x|}\right)^{\frac{n-2}{n}\lambda} \right), &\text{ if }\beta\geq n-2
    \end{cases}\\
   \label{eq3.30}
    v(x)&=O\left( \left(\frac{1}{|x|}\right)^{n-2}\right) +o\left( \left(\frac{1}{|x|}\right)^\beta \log \frac{1}{|x|}\right).
   \end{align}
\end{description}
 \item[Case D.] Suppose $\sigma>\frac{2}{n-2}$.  Let
 \begin{equation}\label{eq3.31}
  a:=\frac{\lambda}{n}[(n-2)\sigma-2] \quad\text{and} \quad 
b:=\frac{\alpha}{n}[(n-2)\sigma-2]+\beta.
 \end{equation}
 Then $0<a<1$.
 \begin{description}
 \item[(D1)] If either
 \begin{enumerate}
  \item $\frac{b}{1-a}<n-2$ and $(n-2)\lambda \leq n-\alpha$; or
  \item $\frac{b}{1-a}\geq n-2$ and $\frac{b\lambda}{1-a}<n-\alpha$
 \end{enumerate}
  then, as $x\to 0$, $u$ and $v$ satisfy \eqref{eq3.25} and \eqref{eq3.26}.
 \item[(D2)] If neither (i) nor (ii) holds then as $x\to0$
 \begin{equation}\label{eq3.32}
  u(x)=
  \begin{cases}
   o\left( \left(\frac{1}{|x|}\right)^{\frac{n-2}{n}[(n-2)\lambda+\alpha]} \right), &\text{ if }\frac{b}{1-a}<n-2\\
   o\left( \left(\frac{1}{|x|}\right)^{\frac{n-2}{n} (\frac{b\lambda}{1-a}+\alpha+\varepsilon)} \right), &\text{ if }\frac{b}{1-a}\geq n-2
  \end{cases}
 \end{equation}
 and
 \begin{equation}\label{eq3.33}
  v(x)=
  \begin{cases}
   O \left( \left(\frac{1}{|x|}\right)^{n-2} \right), &\text{ if }\frac{b}{1-a}<n-2\\
   o\left( \left(\frac{1}{|x|}\right)^{\frac{b}{1-a}+\varepsilon} \right), &\text{ if }\frac{b}{1-a}\geq n-2
  \end{cases}
 \end{equation}
 for all $\varepsilon>0$.
\end{description}
\end{description}
\end{thm}

\section{Nonlinear potentials}\label{nonlin}

In this section we are concerned with pointwise and integral estimates
of certain nonlinear potentials using inequalities of Hedberg type and
Wolff potential estimates (see \cite{AH}, \cite{Maz}). As a
consequence, we will prove the following theorem.

\begin{thm}\label{thm1} Let  $B=B_1(0)$ be the unit ball in $\mathbb{R}^n$, $n \ge 3$, and let 
\begin{equation} \label{ub}
N f(x) = \int_B \frac{f(y)}{|x-y|^{n-2}} dy, \quad x \in B. 
\end{equation}
Then, for all nonnegative functions $f \in L^\infty(B)$,
\begin{equation} \label{ub1}
|| N (( Nf)^\sigma) ||_{L^{\infty} (B)} \le C ||f||_{L^s(B)}^{\frac{2s (\sigma+1)}{n}} 
||f||_{L^\infty(B)}^{\frac{(n-2s)\sigma-2s}{n}}, 
\end{equation}
if $ \sigma> \frac{2}{n-2}$ and $0<s< \frac{n \sigma}{2(\sigma+1)}$, and  
\begin{equation} \label{ub2}
|| N (( Nf)^\sigma) ||_{L^{\infty} (B)} \le C ||f||_{L^s(B)}^{\sigma} \log 
\left ( \frac{C \, ||f||_{L^\infty(B)}}{||f||_{L^s (B)}} 
\right), 
\end{equation}
if $ \sigma\ge  \frac{2}{n-2}$ and $s= \frac{n \sigma}{2(\sigma+1)}$, where $C$ is a positive constant which does not depend on $f$. 
\end{thm}

More precise pointwise estimates of $N (( Nf)^\sigma) $ in terms of
Wolff potentials
$$
\mathbf{W}_\sigma f (x)=\int_0^3  \left(\int_{B_r(x)} f(y) dy\right)^{\sigma} \frac{dr}{r^{(n-2)\sigma-1}}, \quad x \in B_1(0),  
$$
along with their analogues for functions $f$ defined on the entire space $\mathbb{R}^n$, and Riesz or Bessel potentials in place 
of $N f$, will be discussed below (see Theorems \ref{thm2}--\ref{thm4}). 

We remark that if $ \frac{2}{n-2}\le \sigma < \frac{n}{n-2}$, then for all nonnegative $f \in L^1(B_1(0))$, 
$$C^{-1} \mathbf{W}_\sigma f (x) \le N (( Nf)^\sigma)(x) \le C \mathbf{W}_\sigma f (x), \quad x \in B_1(0),$$
where $C$ is a positive constant which depends only on $\sigma$ and
$n$. There are similar pointwise estimates in the range $\frac{n}{n-2}
\le \sigma < \infty$ under the additional assumption that $N ((
Nf)^\sigma)(x) $ is uniformly bounded, for instance, if $f \in
L^\infty(B_1(0))$. These relations between nonlinear potentials $N ((
Nf)^\sigma)$ and $\mathbf{W}_\sigma f $ are due to Havin and Maz'ya,
D. Adams and Meyers (see \cite{AH}, \cite{Maz}).

\medskip

Let $\mu $ be a nonnegative Borel measure on $\mathbb{R}^n$.  For
$0<\alpha <n$, the Riesz potential $\I_{\alpha}\mu$ of order $\alpha$
is defined by
\begin{equation} \label{Rieszpote}
\I_{\alpha}\mu (x)=\int_0^{\infty} \frac{\mu(B_r(x))}{r^{n-\alpha}}\frac{dr}{r} =
 \frac{1}{n-\alpha} \int_{\mathbb{R}^n} \frac{d \mu (y)}{|x-y|^{n-\alpha}}, \quad x \in \mathbb{R}^n. 
\end{equation}
For $1<p< \infty$ and $0<\alpha < \frac n p$, the Wolff potential
$\mathbf{W}_{\alpha, p}\mu$ is defined by (see \cite{AH}, \cite{Maz}):
\begin{equation}\label{wolffpot}
\mathbf{W}_{\alpha, p}\mu(x)=\int_0^{\infty} \left(\frac{\mu(B_r(x))}{r^{n-\alpha p}}\right)^{\frac{1}{p-1}}\frac{dr}{r}, \quad x \in \mathbb{R}^n. 
 \end{equation}
 There is also a nonhomogeneous version applicable for $0<\alpha \le \frac n p$, 
 \begin{equation}\label{wolffpotc}
\mathbf{W^c}_{\alpha, p}\mu(x)=\int_0^{\infty} \left(\frac{\mu(B_r(x))}{r^{n-\alpha p}}\right)^{\frac{1}{p-1}} e^{-cr}\frac{dr}{r}, \quad x \in \mathbb{R}^n, 
 \end{equation}
 where $c>0$. Wolff potentials have numerous applications in analysis and PDE (see, for instance, \cite{KM}, \cite{Lab}, \cite{PV}, \cite{Ver}).

We will also use the Havin-Maz'ya potential $\mathbf{U}_{\alpha, p}\mu$, where $1<p<\infty$ and   $0<\alpha < \frac n p$, defined by:  
\begin{equation}\label{havinpot}
\mathbf{U}_{\alpha, p}\mu(x)= \I_\alpha (\I_\alpha \mu)^{\frac{1}{p-1}} (x), \quad x \in \mathbb{R}^n,   
 \end{equation}
along with its nonhomogeneous analogue $ \mathbf{V}_{\alpha, p}\mu$,  where  $1<p<\infty$ and $0<\alpha \le \frac n p$, defined by:  
\begin{equation}\label{havinbesselpot}
 \mathbf{V}_{\alpha, p}\mu(x)= \mathbf{J}_\alpha (\mathbf{J}_\alpha \mu)^{\frac{1}{p-1}} (x),  \quad x \in \mathbb{R}^n. 
 \end{equation}
Here Bessel potentials 
$$\J_\alpha \mu(x) =\int_{\mathbb{R}^n} G_\alpha(x-t) d \mu(t)$$
with Bessel kernels $G_\alpha$, $\alpha>0$,  are used in place of Riesz potentials $\I_\alpha \mu$. 
Clearly, $\J_\alpha \mu (x)\le c_{\alpha, n} \I_\alpha \mu(x)$, and hence $ \mathbf{V}_{\alpha, p}\mu(x) \le c_{\alpha, n}^{\frac 1{p-1}} \mathbf{U}_{\alpha, p}\mu(x)$ for all $x \in \mathbb{R}^n$. 

Note that the Newtonian potential coincides with 
$$ \I_2 \mu(x)= \mathbf{W}_{1,2} \mu(x) = c_n \mathbf{U}_{1,2} \mu (x), \quad n \ge 3.$$
 If $d\mu =f(x)dx$, where $f \ge 0$ and $f \in L^1_{\text{loc}}(\mathbb{R}^n)$, we will denote the corresponding potentials by $\I_{\alpha}f $,  $\mathbf{U}_{\alpha, p} f$, etc.

We will need the Hardy-Littlewood maximal function 
 $$
  \mathbf{M} f(x) = \sup_{r>0} \, \frac{1}{|B_r(x)|} {\int_{B_r(x)} |f(y)| \, d y}, \quad  x \in \mathbb{R}^n.  
$$

Throughout this section  $c$, $c_1$, $c_2$, etc., will stand for constants 
which depend only on $\alpha$, $p$, and $n$.  The following pointwise estimates of nonlinear potentials are due to Havin and Maz'ya, and D. R. Adams. 
(See, for instance, \cite[Sec. 10.4.2]{Maz} and the references therein.)

\begin{thm}\label{thm2} Let  $\mu $ be a locally finite nonnegative Borel measure on $\mathbb{R}^n$. 

\noindent (a) If $1<p<\infty$ and $0<\alpha\le \frac n p$, then, for some $c, c_1>0$, 
\begin{equation} \label{lower}
 \mathbf{V}_{\alpha, p}\mu(x) \ge c_1 \, \mathbf{W^{c}}_{\alpha, p}\mu(x)=
c_1  \int_0^{\infty} \left(\frac{\mu(B_r(x))}{r^{n-\alpha p}}\right)^{\frac{1}{p-1}}e^{-c r} \frac{dr}{r}, \quad x \in \mathbb{R}^n. 
\end{equation}
(b) If $2 - \frac \alpha n<p<\infty$ and $0<\alpha\le \frac n p$, then, for some $c, c_1>0$,  
\begin{equation} \label{upper}
 \mathbf{V}_{\alpha, p}\mu(x) \le c_1\, \mathbf{W^{c}}_{\alpha, p}\mu(x) =c_1 \, \int_0^{\infty} \left(\frac{\mu(B_r(x))}{r^{n-\alpha p}}\right)^{\frac{1}{p-1}}e^{-c r} \frac{dr}{r}, \quad x \in \mathbb{R}^n.
\end{equation}
({c})  If $1<p\le 2 - \frac \alpha n$, $0<\alpha< \frac n p$, and $\phi ({r}) = \sup_{x \in \mathbb{R}^n} \, \mu(B_r(x))$, then    
\begin{equation} \label{upper-phi}
 \mathbf{V}_{\alpha, p}\mu(x) \le c_1 \, \int_0^{\infty} \left(\frac{\phi({r})}{r^{n-\alpha p}}\right)^{\frac{1}{p-1}} e^{-c r}  \frac{dr}{r}, \quad x \in \mathbb{R}^n. 
\end{equation}
(d) The above estimates with $c=0$ hold for the potential $\mathbf{U}_{\alpha, p}\mu$ 
in place of $\mathbf{V}_{\alpha, p}\mu$ if $0<\alpha< \frac{n}{p}$. 
\end{thm}

Theorem \ref{thm2} yields that the Wolff potential $ \mathbf{W}_{\alpha, p}\mu$ is pointwise equivalent to 
the Havin-Maz'ya potential  $ \mathbf{U}_{\alpha, p}\mu$ (and $\mathbf{W^c}_{\alpha, p}\mu$ is equivalent 
to $\mathbf{V}_{\alpha, p}\mu$ if $c>0$, up to a choice of  $c$), provided $2 - \frac \alpha n<p<\infty$. 

In the range $1<p\le 2 - \frac \alpha n$ (which excludes the critical case $\alpha=\frac n p$), 
the sharp upper estimate \eqref{upper} for $\mathbf{U}_{\alpha, p}\mu$ fails, along with its counterpart for  
$ \mathbf{V}_{\alpha, p}\mu$. However, there are natural substitutes  
under the additional assumption that the corresponding nonlinear potential is uniformly bounded. The following theorem is due to 
Adams and Meyer (see \cite[Sec. 10.4.2]{Maz}).

\begin{thm}\label{thm3} Suppose $\mathbf{V}_{\alpha, p}\mu(x) \le K$
  for all $ x \in \mathbb{R}^n$. 

\noindent (a) If $1<p<2 - \frac \alpha n$ and $0<\alpha < \frac n p$, then, for some $c, c_1>0$,  
\begin{equation} \label{upper1}
 \mathbf{V}_{\alpha, p}\mu(x) \le c_1 \, K^{\frac{(2-p)n-\alpha}{n-\alpha p}} \int_0^{\infty} \left(\frac{\mu(B_r(x))}{r^{n-\alpha p}}\right)^{\frac{n-1}{n-\alpha p}}e^{-c r} \frac{dr}{r}, \quad x \in \mathbb{R}^n. 
\end{equation}
(b) If $p=2 - \frac \alpha n$ and $0<\alpha< \frac n p$, then, for some $c, c_1, c_2>0$,   
\begin{equation} \label{upper2}
 \mathbf{V}_{\alpha, p}\mu(x) \le c_1 \, \int_0^{\infty} \left( \frac{\mu(B_r(x))}{r^{n-\alpha p}} 
 \log \left ( \frac {c_{2} K^{p-1} r^{n-\alpha p}}{\mu(B_r(x)} \right)  \right)^{\frac{1}{p-1}} e^{-c r} \frac{dr}{r}, \quad x \in \mathbb{R}^n.
\end{equation} 
 ({c}) The above estimates with $c=0$ hold for the potential $\mathbf{U}_{\alpha, p}\mu$ 
in place of $\mathbf{V}_{\alpha, p}\mu$. 
\end{thm}

We now deduce some pointwise bounds for nonlinear potentials. 

\begin{thm}\label{thm4} Let $p>1$ and $0<\alpha<n$. Then the
  following estimates hold. 
 
\noindent (a) If $2 -\frac \alpha n  < p < \infty$, $0<\alpha< \frac n p$, and $f \in L^s(\mathbb{R}^n)$, where $1 \le s <\frac n {\alpha p}$, then, for all $x \in \mathbb{R}^n$,  
\begin{equation} \label{est1}
 \mathbf{U}_{\alpha, p} f(x) \le c \, \left ( \mathbf{M} f(x)\right)^{\frac{n-\alpha p s}{(p-1)n}} 
 ||f||_{L^s(\mathbb{R}^n)}^{\frac{\alpha p s}{(p-1)n}}.  
\end{equation}
(b) If $1<p<\infty$,  $0<\alpha< \frac n p$, and $f \in L^s(\mathbb{R}^n) \cap L^\infty (\mathbb{R}^n)$, where  $0 < s < \frac n {\alpha p}$, then, for all $x \in \mathbb{R}^n$,   
\begin{equation} \label{est2}
 \mathbf{U}_{\alpha, p} f(x) \le c \, ||f||_{L^\infty(\mathbb{R}^n)}^{\frac{n-\alpha p s}{(p-1)n}} 
 ||f||_{L^s(\mathbb{R}^n)}^{\frac{\alpha p s}{(p-1)n}}.
\end{equation}
({c}) If  $2 -\frac \alpha n < p < \infty$, $0<\alpha \le \frac n p$, and $f \in L^s(\mathbb{R}^n)$, where 
$s =\frac n {\alpha p}$, then, for all $x \in \mathbb{R}^n$,  
\begin{equation} \label{est3}
\mathbf{V}_{\alpha, p}f (x) \le c   ||f||_{L^s(\mathbb{R}^n)}^{\frac{1}{p-1}}  
\left (\frac{(\mathbf{M}f(x))^{\frac{1}{p-1}}}{(\mathbf{M}f(x))^{\frac{1}{p-1}} + ||f||_{L^s(\mathbb{R}^n)}^{\frac{1}{p-1}}}+ \log^{+} \left( \frac{\mathbf{M} f(x)}{||f||_{L^s(\mathbb{R}^n)}} \right) \right).  
\end{equation}
({d}) If $1<p<\infty$, $0< \alpha \le \frac n p$, and $f \in L^s(\mathbb{R}^n) \cap L^\infty (\mathbb{R}^n)$, where $s =\frac n {\alpha p}$,  then, for all $x \in \mathbb{R}^n$,   
\begin{equation} \label{est4}
 \mathbf{V}_{\alpha, p} f(x) \le c   ||f||_{L^s(\mathbb{R}^n)}^{\frac{1}{p-1}} 
 \left ( \frac { ||f||_{L^\infty(\mathbb{R}^n)}^{\frac{1}{p-1}} }    { ||f||_{L^\infty(\mathbb{R}^n)}^{\frac{1}{p-1}} + ||f||_{L^s(\mathbb{R}^n)}^{\frac{1}{p-1}}}+ \log^{+} \left( \frac{ ||f||_{L^\infty(\mathbb{R}^n)} } {||f||_{L^s(\mathbb{R}^n)}} \right) \right). 
\end{equation}
\end{thm}

\begin{proof} Suppose first that $2 - \frac \alpha n <p<\infty$. Fix $R>0$, and let $d \mu = f(x) dx$, where 
$f \ge 0$ and $f \in L^s_{\text{loc}}(\mathbb{R}^n)$, $s \ge 1$.  
 Then by \eqref{upper}, $ \mathbf{V}_{\alpha, p}\mu(x)$ is bounded 
above by a constant multiple of  
$$
\mathbf{W^{c}}_{\alpha, p} f(x)=\int_0^{\infty} \left( \frac{\mu(B_r(x))} {r^{n-\alpha p}} \right)^{\frac{1}{p-1}}e^{-c r} \frac{dr}{r}
$$
$$ 
= \int_0^{R} \left(\frac{\mu(B_r(x))}{r^{n-\alpha p}}\right)^{\frac{1}{p-1}}e^{-c r} \frac{dr}{r} 
+ \int_R^{\infty} \left(\frac{\mu(B_r(x))}{r^{n-\alpha p}}\right)^{\frac{1}{p-1}}e^{-c r} \frac{dr}{r}=I_1+I_2.
$$
Note that 
$$
\frac{\mu(B_r(x))}{r^{n-\alpha p}} \le c_n r^{\alpha p} \mathbf{M}f(x).   
$$
Hence,
$$
I_1 \le C  \mathbf{M} f(x)^{\frac{1}{p-1}}  \int_0^{R} r^{\frac{\alpha p}{p-1}}e^{-c r} \frac{dr}{r}.  
$$
To estimate $I_2$, notice that by H\"older's inequality with $s \ge 1$, 
$$
\mu(B_r(x)) \le ||f||_{L^s(\mathbb{R}^n)} c_n \, r^{n(1-\frac 1 s)}, \quad x \in \mathbb{R}^n, \, \, r>0. 
$$
Hence, 
$$
I_2\le C ||f||_{L^s(\mathbb{R}^n)}^{\frac{1}{p-1}} \int_{R}^{\infty} r^{-\frac{n-\alpha p s}{s(p-1)}} e^{-c r} \frac{dr}{r}.
$$
Letting $a = \frac{\mathbf{M}f(x)}{||f||_{L^s(\mathbb{R}^n)}}$, and combining the preceding inequalities, we obtain
 \begin{equation} \label{upperd1}
\mathbf{V}_{\alpha, p} f (x) \le C ||f||_{L^s(\mathbb{R}^n)}^{\frac{1}{p-1}} \left (a^{\frac{1}{p-1}}  \int_0^{R} r^{\frac{\alpha p}{p-1}}e^{-c r} \frac{dr}{r} 
+   \int_{R}^{\infty} r^{-\frac{n-\alpha p s}{s(p-1)}} e^{-c r} \frac{dr}{r} \right).
\end{equation} 
 Minimizing the right-hand side over $R$ gives, with $R=a^{-\frac{s}{n}}$, 
 \begin{equation} \label{upperd2}
\mathbf{V}_{\alpha, p}f (x) \le C  ||f||_{L^s(\mathbb{R}^n)}^{\frac{1}{p-1}} \left (a^{\frac{1}{p-1}}  \int_0^{a^{-\frac{s}{n}}} r^{\frac{\alpha p}{p-1}}e^{-c r} \frac{dr}{r} 
+   \int_{a^{-\frac{s}{n}}}^{\infty} r^{-\frac{n-\alpha p s}{s(p-1)}} e^{-c r} \frac{dr}{r}\right).
\end{equation} 

As noted above, the preceding inequality holds for $\mathbf{U}_{\alpha, p}f$ in place of $\mathbf{V}_{\alpha, p}f $ if we set $c=0$: 
$$
\mathbf{U}_{\alpha, p}f (x) \le C  ||f||_{L^s(\mathbb{R}^n)}^{\frac{1}{p-1}} \left (a^{\frac{1}{p-1}}  \int_0^{a^{-\frac{s}{n}}} r^{\frac{\alpha p}{p-1}} \frac{dr}{r} 
+   \int_{a^{-\frac{s}{n}}}^{\infty} r^{-\frac{n-\alpha p s}{s(p-1)}}  \frac{dr}{r}\right)   
$$
$$ 
=C_1 ||f||_{L^s(\mathbb{R}^n)}^{\frac{1}{p-1}}  \, a^{\frac{n-\alpha p s}{(p-1)n}}= C_1  \left ( \mathbf{M} f(x)\right)^{\frac{n-\alpha p s}{(p-1)n}} 
 ||f||_{L^s(\mathbb{R}^n)}^{\frac{\alpha p s}{(p-1)n}}, 
$$
provided $p> 2 - \frac \alpha n$, $0<\alpha < \frac n p$ and $1\le s < \frac n {\alpha p}$.  This proves  statement (a) of Theorem \ref{thm4}. 

If $f \in L^s(\mathbb{R}^n) \cap 
L^\infty (\mathbb{R}^n)$, then noticing that $\mathbf{M} f (x) \le ||f||_{L^\infty(\mathbb{R}^n)}$ for every $x \in \mathbb{R}^n$, we deduce from this
 a cruder estimate: 
\begin{equation} \label{upperc}
\mathbf{U}_{\alpha, p}f (x) \le C_1   ||f||_{L^\infty(\mathbb{R}^n)}^{\frac{n-\alpha p s}{(p-1)n}} 
 ||f||_{L^s(\mathbb{R}^n)}^{\frac{\alpha p s}{(p-1)n}}.  
\end{equation} 

In the case $1<p \le 2 - \frac \alpha n$, we have $ \alpha  \le (2-p) n  <\frac n p$. Using 
 \eqref{upper-phi} instead of \eqref{upper}, we see that  \eqref{upperc}  
 still holds. Therefore \eqref{upperc} holds for all $1<p<\frac{n}{\alpha}$ and $1\le s < \frac n {\alpha p}$.  
 
 If $0<s<1$, then obviously 
$||f||_{L^1(\mathbb{R}^n)}\le  ||f||_{L^\infty(\mathbb{R}^n)}^{1-s} \, ||f||_{L^s(\mathbb{R}^n)}^s$. 
Consequently,  \eqref{upperc}   for $0<s<1$ follows from the corresponding 
estimate with $s=1$, which  proves  statement (b) of Theorem \ref{thm4}. 

 To prove statement ({c}), notice that in the case $s =\frac{n}{\alpha p }\ge 1$ and $p> 2 - \frac \alpha n$, 
 we clearly have, by looking at the cases $0<a<1$ and $a \ge 1$,  
 $$
 a^{\frac{1}{p-1}}  \int_0^{a^{-\frac{s}{n}}} r^{\frac{\alpha p}{p-1}}e^{-c r} \frac{dr}{r} 
+   \int_{a^{-\frac{s}{n}}}^{\infty}  e^{-c r} \frac{dr}{r} \le C \left ( \left (\frac{a}{a+1}\right)^{\frac{1}{p-1}}  + \log^{+} a \right). 
 $$
Hence  \eqref{upperd2} yields 
\begin{equation} \label{upperd}
 \mathbf{V}_{\alpha, p}f (x) \le c   ||f||_{L^s(\mathbb{R}^n)}^{\frac{1}{p-1}}  
\left (\frac{(\mathbf{M}f(x))^{\frac{1}{p-1}}}{(\mathbf{M}f(x))^{\frac{1}{p-1}} + ||f||_{L^s(\mathbb{R}^n)}^{\frac{1}{p-1}}}+ \log^{+} \left( \frac{\mathbf{M} f(x)}{||f||_{L^s(\mathbb{R}^n)}} \right) \right), 
\end{equation}
for $s =\frac{n}{\alpha p}\ge 1$ and $p> 2 - \frac \alpha n$, which proves statement ({c}).  
 
 In the case $s =\frac{n}{\alpha p}>1$ and $1<p \le 2 - \frac \alpha n$, as in the 
 estimates of  $\mathbf{U}_{\alpha, p}f$ above, we again use \eqref{upper-phi} in place of \eqref{upper}, 
 together with $\mathbf{M}f(x)\le ||f||_{L^\infty(\mathbb{R}^n)}$, to complete the proof of statement (d). 
 \end{proof}

 \begin{rmk} {\rm Inequality \eqref{est2} can be deduced directly
     applying Hedberg's inequality (see \cite{AH}, Proposition
     3.1.2(a)) to $\mathbf{I}_\alpha g$ where $g=(\mathbf{I}_\alpha
     f)^{\frac 1{p-1}}$, followed by Sobolev's inequality.}
\end{rmk}

\begin{rmk} {\rm Theorem \ref{thm1} is immediate from statements (b) and (d) of Theorem \ref{thm4} 
with $\alpha=2$ and $p=1+ \frac 1 \sigma$.} 
 \end{rmk}

We will need the following corollary of Theorem \ref{thm1} in the next
section. 

\begin{cor}\label{cor1}
  Let $g\in L^\infty(B)$ be a nonnegative function where $B=B_R(x_0)$ is a ball
  in $\mathbb{R}^n$, $n\ge 3$, and let
\[
Ng(x)=\int_B\frac{g(y)}{|x-y|^{n-2}}\,dy,\quad x\in B.
\]
Then for $\sigma\ge\frac{2}{n-2}$ we have
\[
\|N((Ng)^\sigma)\|_{L^\infty(B)}\le\begin{cases}
C \|g\|_{L^1(B)}^{\frac{2\sigma+2}{n}} 
||g||_{L^\infty(B)}^{\frac{(n-2)\sigma-2}{n}},&\text{if $\sigma>\frac{2}{n-2}$}\\
C\|g\|_{L^1(B)}^{\sigma} \log 
\left( \frac{C \,|B|\, ||g||_{L^\infty(B)}}{||g||_{L^1 (B)}}\right),
&\text{if $\sigma=\frac{2}{n-2}$} 
\end{cases}
\]
where $C=C(n,\sigma)$ is a positive constant.
\end{cor}

\begin{proof}
Apply Theorem \ref{thm1} with $s=1$ to the function $f(x)=g(x_0+Rx)$.
\end{proof}

\section{Preliminary lemmas}\label{sec4}
In this section we provide some lemmas needed for the proofs of our
results in Sections \ref{sec2} and \ref{sec3}. 

\begin{lem}\label{lem4.1}
  Let $\varphi:(0,1)\to(0,1)$ be a continuous
  function such that $\lim_{r\to0^{+}} \varphi(r)=0$.  Let $\{ x_j
  \}^{\infty}_{j=1}$ be a sequence in $\mathbb{R}^n$, where $n\geq 3$
  (resp. $n=2$), such that
 \begin{equation}\label{eq4.1}
  0<4|x_{j+1}|<|x_j|<\frac{1}{2}
 \end{equation}
and
 \begin{equation}\label{eq4.2}
  \sum^{\infty}_{j=1}\varphi(|x_j|)<\infty.
 \end{equation}
 Let $\{r_j \}^{\infty}_{j=1} \subset \mathbb{R}$ be a sequence satisfying
 \begin{equation}\label{eq4.3}
  0<r_j \leq |x_j| /2.
 \end{equation}
 Then there exist a positive constant $A=A(n)$ and a positive function
 $u\in C^\infty (\Omega\backslash \{0\})$ where $\Omega=\mathbb{R}^n$ 
(resp. $\Omega=B_2 (0)\subset \mathbb{R}^2$) such that
 \begin{equation}\label{eq4.4}
  0\leq -\Delta u\leq \frac{\varphi(|x_j |)}{r_{j}^n} \quad\text{in}
\quad B_{r_j} (x_j)
 \end{equation}
 \begin{equation}\label{eq4.5}
  -\Delta u=0\quad \text{ in }\quad\Omega \backslash \left( \{0\} \cup \bigcup^{\infty}_{j=1} B_{r_j} (x_j) \right)
 \end{equation}
 \begin{equation}\label{eq4.6}
  u\geq \frac{A\varphi(|x_j |)}{r_{j}^{n-2}}\quad\left( \text{resp. }
    u\geq A\varphi(|x_j |)\log \frac{1}{r_j} \right) \quad\text{ in } \quad
B_{r_j} (x_j )
 \end{equation}
 \begin{equation}\label{eq4.7}
  u\geq 1\quad\text{ in }\Omega\backslash \{0\}.
 \end{equation}
\end{lem}

\begin{proof}
 Let $\psi:\mathbb{R}^n \to [0,1]$ be a $C^\infty$ function whose
 support is $\overline{B_1 (0)}$.  Define $\psi_j :\mathbb{R}^n \to
 [0,1]$ by 
$\psi_j (y)=\psi(\eta)$ where $y=x_j +r_j \eta$.
 Then
 \begin{equation}\label{eq4.8}
  \int_{\mathbb{R}^n} \psi_j (y)\,dy=\int_{\mathbb{R}^n} \psi(\eta)r^{n}_{j} \,d\eta=r^{n}_{j}I
 \end{equation}
 where $I=\int_{\mathbb{R}^n}\psi(\eta)\,d\eta>0$.  Let $\varepsilon_j :=\varphi(|x_j|)$ and
 \begin{equation}\label{eq4.9}
  f:=\sum^{\infty}_{j=1}M_j \psi_j \quad\text{ where }M_j=\frac{\varepsilon_j}{r^{n}_{j}}.
 \end{equation}
 Since the functions $\psi_j$ have disjoint supports, $f\in C^\infty (\mathbb{R}^n \backslash\{0\})$ and by \eqref{eq4.9}, \eqref{eq4.8}, and \eqref{eq4.2} we have
 \begin{align}\label{eq4.10}
  \notag \int_{\mathbb{R}^n} f(y)\,dy&=\sum^{\infty}_{j=1}M_j \int_{\mathbb{R}^n} \psi_j (y)\,dy=I\sum^{\infty}_{j=1} M_j r^{n}_{j}\\
  &=I\sum^{\infty}_{j=1}\varepsilon_j <\infty.
 \end{align}
\begin{description}
 \item[Case I.] Suppose $n\geq 3$.  Then for $x=x_j +r_j \xi$ and $|\xi|<1$ we have
 \begin{align*}
  \int_{\mathbb{R}^n} &\frac{f(y)}{|x-y|^{n-2}}\,dy\geq \int_{|y-x_j
    |<r_j} 
\frac{M_j \psi_j (y)}{|x-y|^{n-2}} \,dy=\int_{|\eta|<1} 
\frac{M_j \psi(\eta)r^{n}_{j}}{r^{n-2}_{j} |\xi-\eta|^{n-2}}\,d\eta\\
  &=\frac{\varepsilon_j}{r^{n-2}_j} \int_{|\eta|<1} \frac{\psi(\eta)}{|\xi-\eta|^{n-2}} \,d\eta\\
  &\geq \frac{J\varepsilon_j}{r^{n-2}_{j}} \quad\text{ where }J=\min_{|\xi|\leq 1} \int_{|\eta|<1} \frac{\psi(\eta)\,d\eta}{|\xi-\eta|^{n-2}}>0.
 \end{align*}
 Thus letting
 \begin{equation*}
  u(x):=\int_{\mathbb{R}^n} \frac{B}{|x-y|^{n-2}} f(y)\,dy+1 \quad\text{ for }x\in\mathbb{R}^n \backslash \{0\}
 \end{equation*}
 where $\frac{B}{|x|^{n-2}}$ is a fundamental solution of $-\Delta$ we
 have $u$ satisfies \eqref{eq4.6} with $A=BJ$ and
 \begin{equation*}
  -\Delta u(x)=f(x)=M_j \psi_j (x)\leq \frac{\varepsilon_j}{r^{n}_{j}} \quad\text{ for }x\in B_{r_j} (x_j).
 \end{equation*}
 Also $u\in C^\infty (\mathbb{R}^n \backslash \{0\})$ and $u$ clearly
 satisfies \eqref{eq4.5} and \eqref{eq4.7}.
 \item[Case II.] Suppose $n=2$.  Then for $x=x_j +r_j \xi$ and $|\xi|<1$ we have
 \begin{align*}
  \int_{|y|<2} &\left( \log \frac{4}{|x-y|} \right) f(y)\,dy\geq \int_{|y-x_j|<r_j} \left( \log \frac{4}{|x-y|} \right) M_j \psi_j (y)\,dy\\
  &=\int_{|\eta|<1} \left(\log \frac{1}{r_j}+\log \frac{4}{|\xi-\eta|} \right) M_j \psi(\eta)r^{2}_j \,d\eta\\
  &\geq \varepsilon_j \left[ \log\frac{1}{r_j} \int_{|\eta|<1} \psi(\eta)\,d\eta \right] =I\varepsilon_j \log\frac{1}{r_j}.
 \end{align*}
 Thus letting
 \begin{equation*}
  u(x):=\int_{|y|<2} \frac{1}{2\pi} \left(\log \frac{4}{|x-y|} \right) f(y)\,dy+1 \quad\text{for }x\in B_2 (0)\backslash \{0\}
 \end{equation*}
 we have $u$ satisfies \eqref{eq4.6} with $A=\frac{I}{2\pi}$ and
 \begin{equation*}
  -\Delta u(x)=f(x)=M_j \psi_j (x)\leq \frac{\varepsilon_j}{r^{2}_{j}} 
\quad\text{for }x\in B_{r_j}(x_j).
 \end{equation*}
 Also $u\in C^\infty (B_2 (0)\backslash \{0\})$ and $u$ clearly
 satisfies \eqref{eq4.5} and \eqref{eq4.7}.
\end{description}
\end{proof}

\begin{lem}\label{lem4.2}
 If $R>0$ and $x_0 \in\mathbb{R}^n , \, n\geq 3$, then
 \begin{equation}\label{eq4.11}
  \int_{|y-x_0|<R} \frac{\,dy}{|x-y|^{n-2}} \leq \frac{CR^n}{|x-x_0|^{n-2}+R^{n-2}}
 \end{equation}
 for all $x\in \mathbb{R}^n$ where $C=C(n)>0$.
\end{lem}
 
\begin{proof}
 Denote the left side of \eqref{eq4.11} by $N(x)$.  Let $x=x_0 +R\xi$ and $y=x_0 +R\eta$.  Then
 \begin{align*}
  N(x)&=N(x_0 +R\xi)=\int_{|\eta|<1} \frac{R^n \,d\eta}{R^{n-2} |\xi-\eta|^{n-2}}\\
  &\leq \frac{CR^2}{|\xi|^{n-2}+1}=\frac{CR^2R^{n-2}}{|R\xi|^{n-2} +R^{n-2}}\\
  &=\frac{CR^n}{|x-x_0|^{n-2}+R^{n-2}}.
 \end{align*}
\end{proof}

\begin{lem}\label{lem4.3}
 Let $u$ be a $C^2$ nonnegative superharmonic function in $B_{2\varepsilon}(0)\backslash \{0\} \subset \mathbb{R}^2$ satisfying
 \begin{equation}\label{eq4.12}
  \log^{+} (-\Delta u(x))=O\left(H
    \left(\log\frac{1}{|x|}\right) 
\right) \quad\text{ as }x\to0
 \end{equation}
 where $\varepsilon\in(0,1/2)$ and $H:(0,\infty)\to (0,\infty)$ is a
 continuous nondecreasing function satisfying
 $\lim_{t\to\infty}H(t)=\infty$.  Then
 \begin{equation}\label{eq4.13}
  u(x)=O\left(\log\frac{1}{|x|}\right)
+o\left(H\left(\log\frac{2}{|x|}\right) \right) \quad\text{ as }x\to0.
 \end{equation}
\end{lem}

\begin{proof}
  Let $x_j \in B_{\frac{\varepsilon}{2}}(0)\backslash \{0\}$ be a
  sequence which converges to the origin.  It suffices to prove
  \eqref{eq4.13} with $x$ replaced with $x_j$.
 
 By \eqref{eq4.12} there exists $A>0$ such that 
 \begin{equation*}
 \log^{+}(-\Delta u(y))\leq AH\left(\log\frac{2}{|x_j|}\right) \quad\text{ for }|y-x_j|\leq \frac{|x_j|}{2}.
 \end{equation*}
 Thus
 \begin{equation}\label{eq4.13.1}
  0\leq -\Delta u(y)\leq \exp \left(AH\left(\log\frac{2}{|x_j|}\right) \right) \quad\text{ for }|y-x_j| \leq \frac{|x_j|}{2}.
 \end{equation}
 Define $r_j \ge 0$ by
 \begin{equation*}
  \int_{|y-x_j|<r_j} e^{AH\left(\log\frac{2}{|x_j|}\right)}
  dy=\int_{|y-x_j|<\frac{|x_j|}{2}} -\Delta u(y)\,dy\to0 \quad \text{ as }j\to\infty
 \end{equation*}
 by Lemma \ref{A.1}. Thus
 \begin{equation}\label{eq4.14}
  r_j =o\left( \exp \left(-\frac{A}{2} H\left(\log\frac{2}{|x_j|}\right) \right) \right) \quad\text{ as }j\to\infty
 \end{equation}
and by \eqref{eq4.13.1}
\[
\int_{|y-x_j|<\frac{|x_j|}{2}} \left(\log\frac{1}{|y-x_j|}\right) 
(-\Delta u(y))\,dy
\le\int_{|y-x_j|<r_j} \left(\log\frac{1}{|y-x_j|}\right) 
\exp\left(AH\left(\log\frac{2}{|x_j|}\right) \right)\,dy.
\] 
Hence by Lemma \ref{A.1} we get
 \begin{align*}
  u(x_j) &\leq C\left[\log \frac{1}{|x_j|}
+\int_{|y-x_j|>\frac{|x_j|}{2}, |y|<\varepsilon}\left(\log\frac{1}{|y-x_j|} \right) (-\Delta u(y))\,dy\right]\\
 &\quad +C\int_{|y-x_j|<r_j} \left(\log\frac{1}{|y-x_j|}\right) 
\exp\left(AH\left(\log\frac{2}{|x_j|}\right) \right)\,dy\\
  &\leq C\left[ \log\frac{1}{|x_j|}+r^{2}_{j} \left(\log\frac{1}{r_j}\right) \exp\left( AH\left(\log\frac{2}{|x_j|}\right) \right) \right]\\
  &\leq C\log\frac{1}{|x_j|}
+o\left(H\left(\log\frac{2}{|x_j|}\right) \right) \quad \text{as }j\to\infty
 \end{align*}
 by \eqref{eq4.14}.
\end{proof}

\begin{lem}\label{lem4.4}
 Let $u$ be a $C^2$ nonnegative function in $B_{3\varepsilon}(0)\backslash \{0\}\subset \mathbb{R}^n ,\, n\geq 3$, satisfying 
 \begin{equation}\label{eq4.15}
  0\leq -\Delta u(x)=O\left( \left(\frac{1}{|x|}\right)^\gamma \left(\log \frac{1}{|x|}\right)^q \right) \quad\text{ as }x\to0
 \end{equation}
 where $\varepsilon\in (0,1/8)$,  $\gamma\in\mathbb{R}$, and $q\geq 0$ are constants.
 \begin{enumerate}
  \item If $q=0$ then
  \begin{equation}\label{eq4.16}
   u(x)=O\left( \left(\frac{1}{|x|}\right)^{n-2} \right)+o\left( \left(\frac{1}{|x|}\right)^{\gamma\frac{n-2}{n}} \right) \quad\text{ as }x\to0.
  \end{equation}
 \item If $q>0$ and $\gamma\geq n$ then
 \begin{equation}\label{eq4.17}
  u(x)=o\left( \left(\frac{1}{|x|}\right)^{\gamma \frac{n-2}{n}} \left(\log \frac{1}{|x|}\right)^{q\frac{n-2}{n}}\right) \quad\text{ as }x\to0.
 \end{equation}
 \item If $q=0$, $\gamma>n$, and $v(x)$ is a $C^2$ nonnegative solution of
 \begin{equation}\label{eq4.18}
  0\leq -\Delta v\leq|x|^{-\beta} \left(u+|x|^{-(n-2)}\right)^\sigma \quad\text{in }B_{3\varepsilon} (0)\backslash \{0\}
 \end{equation}
 where $\beta\in\mathbb{R}$ and $\sigma\ge 2/(n-2)$ then as $x\to0$ we have
\begin{numcases}{v(x)=}
   O \left( \left(\frac{1}{|x|}\right)^{n-2} \right) +o\left(
     \left(\frac{1}{|x|}\right)^{\frac{\gamma}{n}[(n-2)\sigma-2]+\beta}
   \right) &if $\sigma>\frac{2}{n-2}$ \label{eq4.19}\\
   O \left( \left(\frac{1}{|x|}\right)^{n-2} \right)+o\left(
     \left(\frac{1}{|x|}\right)^\beta \log \frac{1}{|x|} \right) &if
   $\sigma=\frac{2}{n-2}$ \label{eq4.20}.
  \end{numcases}
\end{enumerate}
\end{lem}

\begin{proof}
  Let $\{x_j\}\subset B_\varepsilon (0)\backslash \{0\}$ be a sequence
  which converges to the origin.  It suffices to prove \eqref{eq4.16},
  \eqref{eq4.17}, \eqref{eq4.19}, and \eqref{eq4.20} with $x$ replaced
  with $x_j$.
 
  For the proof of part (i) we can assume $\gamma\geq n$ because
  increasing $\gamma$ to $n$ weakens condition \eqref{eq4.15} and does
  not change the estimate \eqref{eq4.16}.
 
 By \eqref{eq4.15} there exists $A>0$ such that
 \begin{equation*}
  -\Delta u(y)<\frac{A}{(2|y|)^\gamma} \left(\log\frac{1}{2|y|}\right)^q \quad\text{ for }0<|y|<2\varepsilon.
 \end{equation*}
 Thus
 \begin{equation}\label{eq4.21.1}
  -\Delta u(y)\leq \frac{A}{|x_j|^\gamma} \left( \log\frac{1}{|x_j|}\right)^q \quad\text{ for }|y-x_j|<\frac{|x_j|}{2}.
 \end{equation}
 Define $r_j\ge 0$ by
 \begin{equation}\label{eq4.21.2}
  \int_{|y-x_j|<r_j} \frac{A}{|x_j|^\gamma}
  \left(\log\frac{1}{|x_j|}\right)^q
  \,dy=\int_{|y-x_j|<\frac{|x_j|}{2}} -\Delta u(y)\,dy\to0
  \quad\text{as } j\to\infty
 \end{equation}
 by Lemma \ref{A.1}.  Then
 \begin{equation}\label{eq4.22}
  r_j =o\left( |x_j|^{\frac{\gamma}{n}}
    \left(\log\frac{1}{|x_j|}\right)^{-\frac{q}{n}}\right)<<|x_j|\quad
  \text{as }j\to\infty
 \end{equation}
 because $\gamma\geq n$ and $q\geq 0$.  Hence by Lemma \ref{A.1} and
 \eqref{eq4.21.1} we have
 \begin{align*}
  u(x_j)&\leq C\left[ \frac{1}{|x_j|^{n-2}}
    +\int_{|y-x_j|>\frac{|x_j|}{2},\,|y|<2\varepsilon} \frac{-\Delta
      u(y)\,dy}{|y-x_j|^{n-2}} +\int_{|y-x_j|<\frac{|x_j|}{2}}
    \frac{-\Delta u(y)\,dy}{|y-x_j|^{n-2}} \right]\\
&\leq C\left[ \frac{1}{|x_j|^{n-2}}
     +\int_{|y-x_j|<\frac{|x_j|}{2}}
    \frac{-\Delta u(y)\,dy}{|y-x_j|^{n-2}} \right]\\
&\leq C \left[ \frac{1}{|x_j|^{n-2}} +\int_{|y-x_j|<r_j} \frac{A\left( \frac{1}{|x_j|}\right)^\gamma \left(\log \dfrac{1}{|x_j|}\right)^q}{|y-x_j|^{n-2}}\,dy\right]\\
  &\leq C \left[ \left(\frac{1}{|x_j|}\right)^{n-2} +\left(\frac{1}{|x_j|}\right)^\gamma \left(\log\frac{1}{|x_j|}\right)^q r^{2}_{j} \right]\\
  &\leq C \left[ \left(\frac{1}{|x_j|}\right)^{n-2} +o\left( \left(\frac{1}{|x_j|}\right)^\gamma \left(\log\frac{1}{|x_j|}\right)^q \right)^{\frac{n-2}{n}} \right]
 \quad\text{as } j\to\infty
\end{align*}
 by \eqref{eq4.22}, which proves parts (i) and (ii).
 
 We now prove part (iii).  For $|x-x_j|<\frac{|x_j|}{4}$ we have by
 \eqref{eq4.18} and Lemma \ref{A.1} that
 \begin{align*}
  \frac{-\Delta v(x)}{|x|^{-\beta}} &\leq \left(u(x)+|x|^{-(n-2)}\right)^\sigma\\ 
&\leq C\left[ \frac{1}{|x_j|^{n-2}} +\int_{|y-x_j|<\frac{|x_j|}{2}} \frac{-\Delta u(y)\,dy}{|y-x|^{n-2}} +\int_{|y-x_j|>\frac{|x_j|}{2}, \, |y|<2\varepsilon} 
  \frac{-\Delta u(y)\,dy}{|y-x|^{n-2}} \right]^\sigma \\
  &\leq C\left[ \frac{1}{|x_j|^{\sigma(n-2)}} 
+\left( \left( N_{B_{\frac{|x_j|}{2}}(x_j)} (-\Delta u)\right) (x)\right)^\sigma \right] 
 \end{align*}
 where $(N_\Omega f)(x):=\int_{\Omega}\frac{f(y)}{|y-x|^{n-2}}\,dy$.
 
 Thus by Lemma \ref{A.1}
 \begin{align}\label{eq4.23}
  \notag v(x_j)&\leq C\left[ \frac{1}{|x_j|^{n-2}} +\int_{|y-x_j|<\frac{|x_j|}{4}} \frac{\Delta v(y)}{|y-x_j|^{n-2}}\,dy+\int_{|y-x_j|>\frac{|x_j|}{4}, \, |y|<2\varepsilon} \frac{-\Delta v(y)}{|y-x_j|^{n-2}}\,dy\right]\\
  &\leq C \left[ \frac{1}{|x_j|^{n-2}} 
+\frac{|x_j|^{-\beta}}{|x_j|^{\sigma(n-2)-2}}
+|x_j|^{-\beta}(H(-\Delta u))(x_j) \right]
 \end{align}
 where $Hf=N_{B_{\frac{|x_j|}{4}}(x_j)} \left( N_{B_{\frac{|x_j|}{2}}(x_j)} f\right)^\sigma$.
 
 \begin{description}
  \item[Case I.] Suppose $(n-2)\sigma>2$.  Then using 
\eqref{eq4.21.1} and  \eqref{eq4.21.2}
 with $q=0$ in Corollary \ref{cor1} we get 
\[
(H(-\Delta u))(x_j)
=o\left( \frac{1}{|x_j|^{\frac{\gamma}{n}(\sigma(n-2)-2)}} \right) 
\quad\text{ as }j\to\infty.
\]
Thus \eqref{eq4.19} follows from \eqref{eq4.23}.
\item[Case II.] Suppose $(n-2)\sigma=2$. Then using \eqref{eq4.21.1},
  \eqref{eq4.21.2}, and \eqref{eq4.22} with $q=0$ in Corollary
  \ref{cor1} we get
\begin{align*}
(H(-\Delta u))(x_j)&\le C\left(r_j^n\frac{A}{|x_j|^\gamma}\right)^\sigma
\log\left(\frac{C|x_j|^n\left(A/|x_j|^\gamma\right)}{r_j^nA/|x_j|^\gamma}\right)\\
&=C|x_j|^{n\sigma-\gamma\sigma}\left(\frac{r_j}{|x_j|}\right)^{n\sigma}
\log\left(C\left(\frac{|x_j|}{r_j}\right)^n\right)\\
&=o\left(|x_j|^{n\sigma-\gamma\sigma}|x_j|^{(\gamma-n)\sigma}
\log\frac{1}{|x_j|^{\gamma-n}}\right)\\
&=o\left(\log\frac{1}{|x_j|}\right) \quad\text{as }j\to\infty.
\end{align*}
 Thus \eqref{eq4.20} follows from \eqref{eq4.23}.
 \end{description} 
\end{proof}

\begin{lem}\label{lem4.5}
 Suppose $u(x)$ and $v(x)$ are $C^2$ nonnegative solutions of the
 system 
 \begin{equation}\label{eq4.25}
  0\leq -\Delta u
 \end{equation}
 \begin{equation}\label{eq4.26}
  0\leq -\Delta v\leq |x|^{-\beta} \left(u+|x|^{-(n-2)}\right)^\sigma
 \end{equation}
 in a punctured neighborhood of the origin in $\mathbb{R}^n$, $n\geq
 3$, where $\beta\in\mathbb{R}$.
\begin{enumerate}
\item If $0\le\sigma<\frac{2}{n-2}$ then
 \begin{equation}\label{eq4.27}
 v(x)=O\left(|x|^{-(n-2)}+|x|^{2-(n-2)\sigma-\beta}\right)\quad
 \text{as } x\to 0. 
 \end{equation}
\item If $\sigma$ and $\lambda$ satisfy \eqref{eq3.21} and 
\begin{equation}\label{eq4.28}
 -\Delta u\leq|x|^{-\alpha} \left(v+|x|^{-(n-2)} \right)^\lambda
 \end{equation}
in a punctured neighborhood of the origin, where
$\alpha\in\mathbb{R}$, then for some $\gamma>n$ we have
\begin{equation}\label{eq4.29}
  -\Delta u(x)=O\left(|x|^{-\gamma}\right) \quad \text{ as }x\to 0.
 \end{equation}
\end{enumerate}
\end{lem}

\begin{proof}
Choose $\varepsilon\in(0,1)$ such that $u(x)$ and $v(x)$ are $C^2$
  nonnegative solutions of the system (\ref{eq4.25}, \ref{eq4.26}) in
    $B_{2\varepsilon}(0)\backslash \{0\}$.
Let  $\{x_j\}^{\infty}_{j=1}$ be a sequence in $\mathbb{R}^n$ such that
\begin{equation*}
   0<4|x_{j+1}|<|x_j|<\varepsilon/2.
  \end{equation*}
It suffices to prove \eqref{eq4.27} and \eqref{eq4.29} with $x$
replaced with $x_j$.

By Lemma \ref{A.1} we have
  \begin{equation}\label{eq4.30}
   \int_{|x|<\varepsilon} -\Delta u(x)\,dx<\infty \quad\text{and}\quad
   \int_{|x|<\varepsilon} -\Delta v(x)\,dx<\infty
  \end{equation}
  and, for $|x-x_j|<\frac{|x_j|}{4}$,
  \begin{equation}\label{eq4.31}
   u(x)\leq C\left[ \frac{1}{|x_j|^{n-2}} +\int_{|y-x_j|<\frac{|x_j|}{2}} \frac{1}{|x-y|^{n-2}} (-\Delta u(y))\,dy\right] 
  \end{equation}
  \begin{equation}\label{eq4.32}
   v(x)\leq C\left[ \frac{1}{|x_j|^{n-2}} +\int_{|y-x_j|<\frac{|x_j|}{2}} \frac{1}{|x-y|^{n-2}} (-\Delta v(y))\,dy\right] 
  \end{equation}
  where $C>0$ does not depend on $j$ or $x$.
  
  By \eqref{eq4.30}, we have as $j\to\infty$ that
  \begin{equation}\label{eq4.33}
   \int_{|y-x_j|<\frac{|x_j|}{2}} -\Delta u(y)\,dy\to0 \quad\text{and}
   \quad \int_{|y-x_j|<\frac{|x_j|}{2}} -\Delta v(y)\,dy\to0.
  \end{equation}
  Define $f_j ,g_j:\overline{B_2 (0)}\to[0,\infty)$ by
  \begin{equation*}
   f_j(\xi)=-r^{n}_{j} \Delta u(x_j +r_j \xi) \quad\text{and}\quad g_j (\xi)=-r^{n}_{j} \Delta v(x_j +r_j \xi)
  \end{equation*}
  where $r_j=|x_j|/4$.  Making the change of variables $y=x_j +r_j
  \zeta$ in \eqref{eq4.33}, \eqref{eq4.32}, and \eqref{eq4.31}  we get
  \begin{equation}\label{eq4.34}
   \int_{|\zeta|<2} f_j (\zeta)\,d\zeta\to0 \quad\text{and}\quad\int_{|\zeta|<2} g_j (\zeta)\,d\zeta\to0,
  \end{equation}
  and
  \begin{equation}\label{eq4.35}
   v(x_j +r_j \xi)\leq \frac{C}{r^{n-2}_{j}} \left[
     1+(N_2g_j)(\xi)\right] \quad\text{for}\quad |\xi|<1
  \end{equation}
  \begin{equation}\label{eq4.36}
   u(x_j +r_j \xi)\leq \frac{C}{r^{n-2}_{j}} \left[
     1+(N_2f_j)(\xi)\right] \quad\text{for}\quad|\xi|<1.
  \end{equation}
 where $(N_Rf)(\xi):=\int_{|\zeta|<R}|\xi-\zeta|^{-(n-2)}f(\zeta)\,\,d\zeta$. 

We now prove part (i). If $\sigma=0$ then part (i) follows from Lemma
\ref{lem4.4}(i). Hence we can assume $0<\sigma<2/(n-2)$. Define
$\varepsilon\in (0,1)$ and $\gamma>0$ by
\[
\sigma=\frac{2}{n-2}(1-\varepsilon)^2\quad \text{and} \quad
\gamma=\frac{n}{n-2}(1-\varepsilon).
\]
It follows from \eqref{eq4.34} and Riesz potential estimates (see
\cite[Lemma 7.12]{GT}) that $N_2f_j\to 0$ in $L^\gamma(B_2(0))$ and
hence 
\[
(N_2f_j)^\sigma\to 0\quad \text{in } L^\frac{n}{2(1-\varepsilon)}(B_2(0)).
\]
Thus by H\"older's inequality
\begin{equation}\label{eq4.37}
\int_{B_1(0)}\Gamma(N_2f_j)^\sigma\,d\xi\le
\|\Gamma\|_\frac{n}{n+2\varepsilon-2} 
\|(N_2f_j)^\sigma\|_\frac{n}{2(1-\varepsilon)}\to 0
\end{equation}  
where $\Gamma$ is given by \eqref{eq1.5}.
By \eqref{eq4.35} and \eqref{eq4.34} we have
\begin{align}\notag
v(x_j)&\le\frac{C}{|x_j|^{n-2}}\left(1+\int_{B_2(0)}\Gamma g_j\,d\xi\right)\\
      &\le\frac{C}{|x_j|^{n-2}}\left(1+\int_{B_1(0)}\Gamma g_j\,d\xi\right)
\label{eq4.38}
\end{align}
and for $|\xi|<1$ it follows from \eqref{eq4.26} and \eqref{eq4.36} that
\begin{align}\notag
g_j(\xi)&=r_j^n(-\Delta v(x_j+r_j\xi))\\
\notag  &\le Cr_j^n|x_j|^{-\beta}\left(u(x_j+r_j\xi)+|x_j|^{-(n-2)}\right)^\sigma\\
\label{eq4.39} &\le C|x_j|^{n-\beta-(n-2)\sigma}
\left(1+((N_2f_j)(\xi))^\sigma\right).
\end{align}
Substituting \eqref{eq4.39} in \eqref{eq4.38} and using
\eqref{eq4.37}, we get
\[
v(x_j)\le C\left(|x_j|^{-(n-2)}+|x_j|^{2-(n-2)\sigma-\beta}\right)
\]
which completes the proof of part (i).

Next we prove part (ii).
Since increasing $\lambda$ and/or $\sigma$ weakens the conditions
   (\ref{eq4.26}, \ref{eq4.28}) on $u$ and $v$ we can assume instead of
   \eqref{eq3.21} that
  \begin{equation}\label{eq4.40}
   \lambda\geq\sigma\geq\frac{2}{n-2}\quad\text{ and }\quad\sigma<\frac{2}{n-2}+\frac{n}{n-2}\frac{1}{\lambda}.
  \end{equation}
Since, for $R\in(0,\frac{1}{2}]$,
  \begin{equation*}
   \int_{2R<|\zeta|<2} \frac{g_j (\zeta)\,d\zeta}{|\xi-\zeta|^{n-2}}
   \leq \frac{1}{R^{n-2}} \int_{|\zeta|<2} g_j (\zeta)\,d\zeta
   \quad\text{for }|\xi|<R
  \end{equation*}
  and
  \begin{equation*}
   \int_{4R<|\eta|<2} \frac{f_j (\eta)\,d\eta}{|\zeta-\eta|^{n-2}} \leq
   \frac{1}{(2R)^{n-2}} \int_{|\eta|<2} f_j (\eta)\,d\eta 
\quad\text{for }|\zeta|<2R,
  \end{equation*}
  it follows from \eqref{eq4.34}, \eqref{eq4.35} and \eqref{eq4.36} that for
  $R\in(0,\frac{1}{2}]$ we have
  \begin{equation*}
   v(x_j +r_j \xi)\leq Cr^{2-n}_{j} \left[ \frac{1}{R^{n-2}} 
+N_{2R}g_j(\xi)\right] \quad\text{for }|\xi|<R
  \end{equation*}
  and
  \begin{equation*}
   u(x_j +r_j \zeta)\leq Cr^{2-n}_{j} \left[ \frac{1}{R^{n-2}} 
+N_{4R}f_j(\zeta)\right] \quad\text{for }|\zeta|<2R
  \end{equation*}
  where $C$ is independent of $\xi$, $\zeta$, $j$, and $R$.
  It therefore follows from (\ref{eq4.26}, \ref{eq4.28}) that for
  $R\in(0,\frac{1}{2}]$ we have
  \begin{align}\label{eq4.41}
   \notag r^{-n}_{j}f_j (\xi)&=-\Delta u(x_j +r_j \xi)\\
   \notag &\leq Cr^{-\alpha}_{j} \left( r^{2-n}_{j} \left[ \frac{1}{R^{n-2}} +(N_{2R} g_j )(\xi) \right] \right)^\lambda\\
   &\leq Cr^{-\alpha-(n-2)\lambda}_{j} \left[
     \frac{1}{R^{(n-2)\lambda}} +\left( (N_{2R}
       g_j)(\xi)\right)^\lambda \right] \quad \text{ for }|\xi|<R,
  \end{align}
  and
  \begin{align*}
   r^{-n}_{j} g_j (\zeta)&=-\Delta v(x_j +r_j \zeta)\\
   &\leq Cr^{-\beta}_{j} \left( r^{2-n}_{j} \left[ \frac{1}{R^{n-2}} +(N_{4R} f_j )(\zeta)\right] \right)^\sigma \\
   &\leq Cr^{-b}_{j} \left[ \frac{1}{R^{(n-2)\sigma}} +\left( (N_{4R} f_j)(\zeta) \right)^\sigma \right] \quad\text{ for }|\zeta|<2R,
  \end{align*}
  where $b=\beta+(n-2)\sigma$.  Thus for $\xi\in\mathbb{R}^n$ we have
  \begin{align*}
   \left( (N_{2R} g_j )(\xi)\right)^\lambda &\leq \left( Cr^{n-b}_{j} N_{2R} \left[ \frac{1}{R^{(n-2)\sigma}} +(N_{4R} f_j )^\sigma \right](\xi)\right)^\lambda\\
   &\leq Cr^{(n-b)\lambda}_{j} \left[ R^{(2-(n-2)\sigma)\lambda} +\left( (M_{4R} f_j )(\xi)\right)^\lambda \right] 
  \end{align*}
  where $M_R f_j :=N_R \left( (N_R f_j)^\sigma \right)$.  Hence by
  \eqref{eq4.41} there exists a positive constant $a$ which depends
  only on $n$, $\alpha$, $\beta$, $\lambda$, and $\sigma$ such that
  \begin{equation}\label{eq4.42}
   f_j (\xi)\leq C\frac{1}{(Rr_j )^a} \left(1+\left( (M_{4R} f_j)(\xi)\right)^\lambda \right) \quad\text{for }|\xi|<R\leq\frac{1}{2}.
  \end{equation}
By \eqref{eq4.40} there exists
$\varepsilon=\varepsilon(n,\lambda,\sigma)\in (0,1)$ such that 
\begin{equation}\label{eq4.43}
\sigma<\frac{n}{n-2+\varepsilon} \quad \text{and}
\quad\sigma<\frac{2-\varepsilon}{n-2+\varepsilon} + 
\frac{n}{n-2+\varepsilon}\frac{1}{\lambda}.
\end{equation}

To prove for some $\gamma>n$ that \eqref{eq4.29} holds with $x=x_j$,
it suffices by the definition of $r_j$ and $f_j$ to show for some
$\gamma>0$ that the sequence
\begin{equation}\label{eq55.0}
\{r_j^\gamma f_j(0)\} \quad \text{is bounded.}
\end{equation}
To prove \eqref{eq55.0} and thereby complete the proof of Lemma
\ref{lem4.5}(ii), we need the following result.

\begin{lem}\label{lem4.6}
Suppose the sequence
\begin{equation}\label{eq4.45}
\{r_j^{\alpha}f_j\} \quad \text{is bounded in } L^p(B_{4R}(0))
\end{equation}
for some constants $\alpha\ge 0$, $p\in [1,\infty)$, and
$R\in(0,\frac{1}{2}]$.  Let $\beta=\alpha\lambda\sigma+a$ where $a$ is
as in \eqref{eq4.42}. Then there exists a constant 
$C_0=C_0 (n, \lambda,\sigma)>0$ such that the sequence
\begin{equation}\label{eq4.45.5}
\{r_j^{\beta}f_j\} \quad \text{is bounded in } L^q(B_{R}(0))
\end{equation}
provided $q\in[1,\infty]$ and
\begin{equation}\label{small}
 \frac{1}{p}-\frac{1}{q}\le C_0.
\end{equation}
\end{lem}

\begin{proof}
It follows from \eqref{eq4.42} that 
\begin{equation}\label{eq55.4}
r_j^\beta f_j(\xi)\le\frac{C}{R^a}\left(1+((M_{4R}(r_j^\alpha
  f_j))(\xi))^\lambda\right) \quad\text{for } |\xi|<R.
\end{equation}

We can assume
\begin{equation}\label{eq55.5}
p\le n/2
\end{equation}
for otherwise it follows from Riesz potential estimates (see 
\cite[Lemma 7.12]{GT}) and
\eqref{eq4.45} that the sequence $\{N_{4R}(r_j^\alpha f_j)\}$ is bounded in
$L^\infty (B_{4R}(0))$ and hence by \eqref{eq55.4} we see that
\eqref{eq4.45.5} holds for all $q\in[1,\infty]$. 

Define $p_2$ by 
\begin{equation}\label{eq4.46}
  \frac{1}{p}-\frac{1}{p_2}=\frac{2-\varepsilon}{n}.
 \end{equation}
where $\varepsilon=\varepsilon(n,\lambda,\sigma)$ is as in
\eqref{eq4.43}. By \eqref{eq55.5}, $p_2\in (p,\infty)$
and by Riesz potential estimates we have
\begin{equation}\label{eq4.47}
  \Vert(N_{4R}f_j)^\sigma \Vert_{p_2/\sigma}
=\Vert N_{4R}f_j \Vert^{\sigma}_{p_2} \leq C\Vert f_j\Vert^{\sigma}_{p}
 \end{equation}
where $\Vert \cdot \Vert_p:=\Vert \cdot \Vert_{L^p (B_{4R}(0))}$. 
Since, by \eqref{eq4.43},
\[
\frac{1}{p_2}=\frac{1}{p}-\frac{2-\varepsilon}{n}\le 1-\frac{2-\varepsilon}{n}
=\frac{n-2+\varepsilon}{n}< \frac{1}{\sigma}
\]
we have 
\begin{equation}\label{eq55.5.5}
p_2/\sigma>1. 
\end{equation}

We can assume 
\begin{equation}\label{eq55.6}
p_2/\sigma\le n/2
\end{equation}
for otherwise by Riesz potential estimates and \eqref{eq4.47} we have
\[
\|M_{4R}(r_j^\alpha f_j)\|_\infty\le C\|(N_{4R}(r_j^\alpha
f_j))^\sigma\|_{p_2/\sigma}\le C\|r_j^\alpha f_j\|_p^\sigma
\]
which is bounded by \eqref{eq4.45}. Hence \eqref{eq55.4} implies
\eqref{eq4.45.5} holds for all $q\in [1,\infty]$.

Define $p_3$ and $q$ by 
\begin{equation}\label{eq4.48}
\frac{\sigma}{p_2}-\frac{1}{p_3}=\frac{2-\varepsilon}{n}
\quad\text{and}\quad q=\frac{p_3}{\lambda}.
\end{equation}
By \eqref{eq55.5.5} and \eqref{eq55.6}, $p_3\in(1,\infty)$ and by
Riesz potential estimates
\begin{align*}
  \Vert \left(M_{4R}f_j\right)^\lambda \Vert_q &=\Vert M_{4R}f_j\Vert^{\lambda}_{p_3}\\
  &\leq C\Vert(N_{4R}f_j)^\sigma \Vert^{\lambda}_{p_2/\sigma} 
\leq C\Vert f_j \Vert^{\lambda\sigma}_{p}
 \end{align*}
by \eqref{eq4.47}. 
It follows therefore from \eqref{eq55.4} that 
\[
\|r_j^\beta f_j\|_{L^q(B_R(0))}
\le\frac{C}{R^a}\left(1+\|r_j^\alpha f_j\|_p^{\lambda\sigma}\right)
\]
which is a bounded sequence by \eqref{eq4.45}. To complete the proof of Lemma
\ref{lem4.6}, it suffices to show
\begin{equation}\label{eq55.7}
\frac{1}{p}-\frac{1}{q}\ge C_0
\end{equation}
for some $C_0=C_0(n,\lambda,\sigma)>0$ because if \eqref{eq4.45.5} holds
for some $q\ge 1$ satisfying \eqref{eq55.7} then it clearly holds for
all $q\ge 1$ satisfying \eqref{small}.

By \eqref{eq4.46} and \eqref{eq4.48} we have 
 \begin{align}\label{eq4.49}
  \notag \frac{1}{p}-\frac{1}{q}&=\frac{1}{p}-\frac{\lambda}{p_3}
=\frac{1}{p}+\frac{(2-\varepsilon)\lambda}{n}-\frac{\lambda\sigma}{p_2}\\
  \notag &=\frac{1}{p}+\frac{(2-\varepsilon)\lambda}{n}+\frac{(2-\varepsilon)\lambda\sigma}{n}-\frac{\lambda\sigma}{p}\\
  &=-\frac{\lambda\sigma-1}{p}+\frac{(2-\varepsilon)\lambda\sigma+(2-\varepsilon)\lambda}{n}.
 \end{align}
 \begin{description}
  \item[Case I.] Suppose $\lambda\sigma\leq 1$.  Then by \eqref{eq4.49}
 and \eqref{eq4.40}
  \begin{equation*}
   \frac{1}{p}-\frac{1}{q}\geq \frac{(2-\varepsilon)\lambda\sigma+(2-\varepsilon)\lambda}{n} \geq C_1(n)>0.
  \end{equation*}
  \item[Case II.] Suppose $\lambda\sigma>1$.  Then, by \eqref{eq4.49},
  \begin{align*}
   \frac{1}{p}-\frac{1}{q}&\geq 1-\sigma\lambda+
\frac{(2-\varepsilon)\lambda\sigma+(2-\varepsilon)\lambda}{n}\\
   &=\frac{1}{n} [n+(2-\varepsilon)\lambda-\lambda\sigma(n-(2-\varepsilon))]\\
   &=\frac{(n-(2-\varepsilon))\lambda}{n} 
\left[ \frac{2-\varepsilon}{n-(2-\varepsilon)}+\frac{n}{n-(2-\varepsilon)} \frac{1}{\lambda}-\sigma \right]\\
   &=C_2(n,\lambda,\sigma)>0
  \end{align*}
  by \eqref{eq4.43}.  
\end{description}
Thus \eqref{eq55.7} holds with $C_0=\min(C_1,C_2)$.
This completes the proof of Lemma \ref{lem4.6}.
\end{proof}

We return now to the proof of Lemma \ref{lem4.5}(ii).  By
\eqref{eq4.34}, the sequence $\{f_j\}$ is bounded in $L^1(B_2(0))$.
Starting with this fact and iterating Lemma \ref{lem4.6} a finite
number of times ($m$ times is enough if $m>1/C_0$) we see that there
exists $R_0 \in(0,\frac{1}{2})$ and $\gamma>n$ such that sequence
$\{r_j^\gamma f_j\}$ is bounded in $L^\infty(B_{R_0}(0))$. In
particular \eqref{eq55.0} holds. This completes the proof of Lemma
\ref{lem4.5}(ii).
\end{proof}

\section{Proofs of two dimensional results}\label{sec5}

In this section we prove Theorems \ref{thm2.1}--\ref{thm2.5}.  The
following theorem with $h(t)=t^\lambda$ immediately implies Theorems
\ref{thm2.1} and \ref{thm2.3}.  We stated Theorems \ref{thm2.1} and
\ref{thm2.3} separately in order to clearly highlight the differences
between possibilities (i) and (iii) which are stated at the beginning
of Section \ref{sec2}.

\begin{thm}\label{thm5.1}
 Suppose $u(x)$ and $v(x)$ are $C^2$ positive solutions of the system
 \begin{equation}\label{eq5.1}
  0\leq -\Delta u
 \end{equation}
 \begin{equation}\label{eq5.2}
  0\leq -\Delta v\leq g(u) 
 \end{equation}
 in a punctured neighborhood of the origin in $\mathbb{R}^2$, where
 $g:(0,\infty)\to(0,\infty)$ is a continuous function satisfying
 \begin{equation}\label{eq5.3}
 \log^{+} g(t)=O(t)\quad\text{ as }t\to\infty.
 \end{equation}
 Then $v$ is harmonically bounded, that is
 \begin{equation}\label{eq5.4}
  \limsup_{x\to 0}\frac{v(x)}{\log\frac{1}{|x|}}<A
 \end{equation}
 for some constant $A>0$. 

If, in addition,
 \begin{equation}\label{eq5.5}
  -\Delta u\leq f(v)
 \end{equation}
 in a punctured neighborhood of the origin, where
 $f:(0,\infty)\to(0,\infty)$ is a continuous function satisfying
 \begin{equation}\label{eq5.55}
  \log^{+} f(t)=O(h(t))\quad\text{ as }t\to\infty
 \end{equation}
 for some continuous nondecreasing function $h:(0,\infty)\to(0,\infty)$
 satisfying $\lim_{t\to\infty}h(t)=\infty$ then
 \begin{equation}\label{eq5.6}
  u(x)=O\left( \log\frac{1}{|x|} \right) 
+o\left(h\left(A\log\frac{2}{|x|}\right)\right) \quad\text{ as }x\to0.
 \end{equation}
\end{thm}

For simplicity and to motivate Theorem \ref{thm2.5}, we stated Theorem
\ref{thm2.3} for the special case $h(t)=t^\lambda$ rather than for
more general $h$ as in Theorem \ref{thm5.1}. Also, the bound
\eqref{eq2.9} in Theorem \ref{thm2.3} is optimal by Theorem
\ref{thm2.4}, whereas in general we can only show the bound
\eqref{eq5.6} in Theorem \ref{thm5.1} is essentially optimal (see
Theorem \ref{thm6.2}).

\begin{proof}[Proof of Theorem \ref{thm5.1}]
Since $u$ is positive and superharmonic in a punctured neighborhood
of the origin, there exists a constant $\varepsilon \in(0,1/4)$ such
that $u>\varepsilon$ in $B_{2\varepsilon}(0)\backslash \{0\}$.
Choose a positive constant $K$ such that $g(t)\leq e^{Kt}$ for
$t>\varepsilon$.  Then $v$ is a $C^2$ positive solution of
\begin{equation}\label{eq5.7}
0\leq -\Delta v\leq e^{Ku} \quad\text{ in }B_{2\varepsilon} (0)\backslash \{0\}.
\end{equation}

   Since $u$ and $v$ are positive and superharmonic in $B_{2\varepsilon}(0)\backslash \{0\}$, we have by Lemma \ref{A.1} that
   \begin{equation}\label{eq5.8}
    -\Delta u, \, -\Delta v\in L^1 (B_\varepsilon (0))
   \end{equation}
   and
    \begin{equation}\label{eq5.9}
     \begin{split} 
    u(x)=m_1 \log \frac{1}{|x|}+\frac{1}{2\pi} \int_{|y|<\varepsilon} \left( \log\frac{1}{|x-y|}\right) (-\Delta u(y))\,dy+h_1 (x)\\
    v(x)=m_2 \log \frac{1}{|x|}+\frac{1}{2\pi} \int_{|y|<\varepsilon} \left( \log\frac{1}{|x-y|}\right) (-\Delta v(y))\,dy+h_2 (x)
     \end{split}
    \quad\text{for } 0<|x|<\varepsilon
    \end{equation}
    where $m_1$, $m_2\ge 0$ are constants and $h_1$,$h_2
    :B_\varepsilon (0)\to\mathbb{R}$ are harmonic functions.
   
   Suppose for contradiction there exists a sequence $\{x_j \}^{\infty}_{j=1} \subset B_{\frac{\varepsilon}{2}} (0)\backslash \{0\}$ such that $x_j\to0$ as $j\to\infty$ and
   \begin{equation}\label{eq5.11}
    \frac{v(x_j)}{\log\frac{1}{|x_j|}} \to\infty \quad\text{ as }j\to\infty.
   \end{equation}
   Since, for $|x-x_j |<\frac{|x_j|}{4}$,
   \begin{equation*}
    \int_{|y-x_j |>\frac{|x_j|}{2}, \, |y|<\varepsilon} \left(\log\frac{1}{|x-y|}\right) (-\Delta u(y))\,dy\leq \left(\log\frac{4}{|x_j|}\right) \int_{|y|<\varepsilon} -\Delta u(y)\,dy,
   \end{equation*}
   and similarly for $v$, it follows from \eqref{eq5.8} and
   \eqref{eq5.9} that
    \begin{equation}\label{eq5.12}
     \begin{split} 
    u(x)\leq C\log\frac{1}{|x_j|}+\frac{1}{2\pi} \int_{|y-x_j|<\frac{|x_j|}{2}} \left(\log\frac{1}{|x-y|}\right) (-\Delta u(y))\,dy\\
    v(x)\leq C\log\frac{1}{|x_j|}+\frac{1}{2\pi} \int_{|y-x_j|<\frac{|x_j|}{2}} \left(\log\frac{1}{|x-y|}\right) (-\Delta v(y))\,dy
     \end{split}
    \quad \text{for } |x-x_j|<\frac{|x_j|}{4}
    \end{equation}   
    where $C$ does not depend on $j$ or $x$.
    
    Substituting $x=x_j$ in \eqref{eq5.12} and using \eqref{eq5.11} we get
    \begin{equation}\label{eq5.14}
     \frac{1}{\log\frac{1}{|x_j|}} \int_{|y-x_j|<\frac{|x_j|}{2}} 
\left(\log\frac{1}{|x_j-y|}\right) (-\Delta v(y))\,dy\to\infty \quad\text{ as }j\to\infty.
    \end{equation}
    Also, \eqref{eq5.8} implies
    \begin{equation}\label{eq5.15}
     \int_{|y-x_j|<\frac{|x_j|}{2}} -\Delta u(y)\,dy\to0 \quad\text{and}
     \quad \int_{|y-x_j|<\frac{|x_j|}{2}} -\Delta v(y)\,dy\to0 \quad\text{as }j\to\infty.
    \end{equation}
    Let $r_j =\frac{|x_j|}{4}$ and define $f_j ,g_j :\overline{B_2
      (0)}\to[0,\infty)$ by 
\[
f_j(\zeta)=-r^{2}_{j} \Delta u(x_j +r_j \zeta) \quad \text{and}\quad
    g_j (\zeta)=-r^{2}_{j} \Delta v(x_j +r_j \zeta).  
\]
Making the change of variables $y=x_j +r_j \zeta$ in \eqref{eq5.15},
\eqref{eq5.14}, and \eqref{eq5.12} and using \eqref{eq5.7} we get
    \begin{equation}\label{eq5.16}
     \int_{|\zeta|<2} f_j (\zeta)\,d\zeta\to0 \quad\text{ and }\quad\int_{|\zeta|<2} g_j (\zeta)\,d\zeta\to0 \quad\text{ as }j\to\infty
    \end{equation}
    \begin{equation}\label{eq5.17}
     \frac{1}{M_j} \int_{|\zeta|<2} \left(\log\frac{4}{|\zeta|}\right) g_j (\zeta)\,d\zeta\to\infty \quad\text{ as }j\to\infty
    \end{equation}
    and
    \begin{align}\label{eq5.18}
     \notag g_j (\xi)&\leq -\Delta v(x_j +r_j \xi)\leq \exp(Ku(x_j +r_j \xi))\\
     &\leq \exp\left(M_j +\frac{K}{2\pi} \int_{|\zeta|<2} \left(\log\frac{4}{|\xi-\zeta|}\right) f_j (\zeta)\,d\zeta\right) \quad\text{ for }|\xi|<1
    \end{align}
    where $M_j =C\log\frac{1}{|x_j|}$ and $C$ does not depend on $j$ or $\xi$.
    
    Let $\Omega_j =\{ \xi\in B_1 (0):u_j (\xi)>M_j\}$ where
    \begin{equation*}
     u_j (\xi):=\frac{K}{2\pi} \int_{|\zeta|<2} \left(\log\frac{4}{|\xi-\zeta|}\right) f_j (\zeta)\,d\zeta.
    \end{equation*}
    Then letting $p_j =\pi/(K\int_{|\zeta|<2} f_j (\zeta)\,d\zeta$), it
    follows from \eqref{eq5.18} that
    \begin{align*}
     \int_{\Omega_j} g_j (\xi)^{p_j} d\xi &\leq \int_{|\xi|<2} e^{2p_j u_j (\xi)} d\xi\\
     &\leq \int_{|\xi|<2} \left(\int_{|\zeta|<2} \frac{4}{|\xi-\zeta|} \frac{f_j (\zeta)}{\int_{|\zeta|<2} f_j} \,d\zeta \right) d\xi, \text{ by Jensen's inequality,}\\
     &\leq 16\pi, \text{ by interchanging the order of integration.}
    \end{align*}
    (The idea of using Jensen's inequality as above is due to Brezis
    and Merle \cite{BM}.) Thus by \eqref{eq5.16} and H\"older's inequality
    \begin{equation*}
     \limsup_{j\to\infty} \int_{\Omega_j} \left(\log\frac{4}{|\zeta|}\right) g_j (\zeta)\,d\zeta<\infty.
    \end{equation*}
    Hence, defining $\hat{g}_j :B_1 (0)\to[0,\infty)$ by
    \begin{equation*}
     \hat{g}_j (\xi)=
     \begin{cases}
      g_j (\xi), &\text{ for }\xi\in B_1 (0)\backslash\Omega_j\\
      0, &\text{ for }\xi\in\Omega_j
     \end{cases}
    \end{equation*}
    it follows from \eqref{eq5.16} and \eqref{eq5.17} that
    \begin{equation}\label{eq5.19}
     \frac{1}{M_j} \int_{|\zeta|<1} \left(\log\frac{4}{|\zeta|}\right) \hat{g}_j (\zeta)\,d\zeta\to\infty \quad\text{ as }j\to\infty.
    \end{equation}
    By \eqref{eq5.16} and \eqref{eq5.18} we have
    \begin{equation}\label{eq5.20}
     \int_{|\zeta|<1} \hat{g}_j (\zeta)\,d\zeta\to0 \quad\text{ as }j\to\infty
    \end{equation}
    and
    \begin{equation}\label{eq5.21}
     \hat{g}_j (\xi)\leq e^{2M_j} \quad\text{ in }B_1 (0).
    \end{equation}

    For fixed $j$, think of $\hat{g}_j (\zeta)$ as the density of a
    distribution of mass in $B_1 (0)$ satisfying \eqref{eq5.19},
    \eqref{eq5.20}, and \eqref{eq5.21}.  By moving small pieces of
    this mass nearer to the origin in such a way that the new density
    (which we again denote by $\hat{g}_j (\zeta)$) does not violate
    \eqref{eq5.21}, we will not change the total mass
    $\int_{|\zeta|<1}\hat{g}_j (\zeta)\,d\zeta$ but $\int_{|\zeta|<1}
    (\log(4/|\zeta|))\hat{g}_j (\zeta)\,d\zeta$ will increase.  Thus
    for some $\rho_j\in(0,1)$ the functions
    \begin{equation*}
     \hat{g}_j (\zeta)=
     \begin{cases}
      e^{2M_j} , &\text{ for }|\zeta|<\rho_j\\
      0, &\text{ for }\rho_j <|\zeta|<1
     \end{cases}
    \end{equation*}
    satisfy \eqref{eq5.19}, \eqref{eq5.20}, and \eqref{eq5.21} which,
    as elementary and explicit calculations show, is impossible
    because $M_j \to\infty$ as $j\to\infty$.  This contradiction
    proves \eqref{eq5.4}.
 
Since $v(x)$ is positive and superharmonic, $v$ is bounded below in
some punctured neighborhood of the origin by some constant $\delta\in
(0,1)$. Hence by \eqref{eq5.4} we have
\[
\delta\le v(x)\le A\log\frac{1}{|x|} \quad \text{for $|x|$ small and
  positive.}
\]
Also by \eqref{eq5.55} there exists a positive constant $C$ such that
\[
\log^+f(t)\le Ch(t) \quad \text{for $t\ge \delta$.}
\]
Hence for $|x|$ small and positive we have by \eqref{eq5.5} that
\begin{align*}
\log^+(-\Delta u(x))&\le \log^+f(v(x))\le Ch(v(x))\\
                    &\le Ch\left(A\log\frac{1}{|x|}\right)=CH\left(\log\frac{1}{|x|}\right)
\end{align*}
where $H(t)=h(At)$. Thus \eqref{eq5.6} follows from Lemma \ref{lem4.3}.
\end{proof}

\begin{proof}[Proof of Theorem \ref{thm2.2}]
 Define $F,M:(0,\infty)\to(0,\infty)$ by
 \begin{equation*}
  F(t)=\min \{f(t),g(t)\} \quad\text{and}\quad M(t)=\min_{\tau\geq t} \frac{\log F(\tau)}{\tau}.
 \end{equation*}
Then $M$ is nondecreasing. By \eqref{eq2.5}, $M(t)\to\infty$ as $t\to\infty$
and there exists $K>0$ such that $F(t)>1$ for $t\ge K$. Thus
 \begin{equation}\label{eq5.22}
  tM(t)\leq \min_{\tau\geq t} \log F(\tau) \quad\text{ for }t\ge K.
 \end{equation}
 Define $\varphi:(0,1)\to(0,1)$ by $\varphi(r)=r$ and let
 $\{x_j\}^{\infty}_{j=1} , \, \{r_j\}^{\infty}_{j=1}$, and $A$ be as
 in Lemma \ref{lem4.1}.  By holding $x_j$ fixed and decreasing $r_j$
 we can assume
\begin{equation}\label{eq5.22.1}
A\varphi(|x_j|)\log\frac{1}{r_j}\ge K,
\end{equation}
\begin{equation}\label{eq5.23}
  A\varphi(|x_j|)M\left(A\varphi(|x_j|)\log \frac{1}{r_j} \right)>2
 \end{equation}
 and
 \begin{equation}\label{eq5.24}
  (h(|x_j|))^2 <A\varphi(|x_j|)\log\frac{1}{r_j}.
 \end{equation}
 Let $\Omega=B_2 (0)$.  By Lemma \ref{lem4.1} there exists a positive function
 $u\in C^\infty (\Omega\backslash \{0\})$ which satisfies
 \eqref{eq4.4}--\eqref{eq4.7}.  By \eqref{eq4.6} and \eqref{eq5.24} we have
 \begin{equation*}
  u(x_j)\neq O(h(|x_j|)) \quad\text{ as }j\to\infty
 \end{equation*}
 which implies \eqref{eq2.6}.  Also for $x\in B_{r_j} (x_j)$ and
 $-\Delta u(x)>0$ it follows from \eqref{eq4.6}, \eqref{eq5.22.1},
 \eqref{eq5.22}, \eqref{eq5.23} and \eqref{eq4.4} that
 \begin{align*}
  \log F(u(x))&\geq \left( A\varphi(|x_j|)\log\frac{1}{r_j}\right) M\left(A\varphi(|x_j|)\log\frac{1}{r_j}\right)\\
  &>2\log\frac{1}{r_j}\\
  &\geq \log(-\Delta u(x)).
 \end{align*}
 Thus $u$ satisfies 
 \begin{equation}\label{eq5.25}
  0\leq -\Delta u\leq F(u)
 \end{equation}
 in $B_{r_j} (x_j)$.  By \eqref{eq4.5}, $u$ satisfies \eqref{eq5.25}
 in $\Omega\backslash(\{0\}\cup \bigcup^{\infty}_{j=1} B_{r_j}
 (x_j))$.  Thus $u$ satisfies \eqref{eq5.25} in $\Omega\backslash
 \{0\}$.  Taking $v=u$ completes to proof of Theorem \ref{thm2.2}.
\end{proof}

The following theorem with $h(t)=t^\lambda$, $\lambda>1$, immediately implies
Theorem \ref{thm2.4}.

\begin{thm}\label{thm6.2}
  Suppose $h:(0,\infty)\to(0,\infty)$ and $\psi:(0,1)\to(0,1)$ are 
  continuous nondecreasing functions satisfying
\begin{equation}\label{equ6.9}
\lim_{t\to\infty}\frac{h(t)}{t}=\infty \quad\text{and}\quad
\lim_{r\to 0^{+}} \psi (r)=0.  
\end{equation}
Then there exist $C^\infty$ positive solutions $u(x)$ and $v(x)$ of
the system
 \begin{equation}\label{equ6.10}
 \begin{aligned}
 0 & \le -\Delta u\le e^{h(v)}\\
 0 & \le -\Delta v\le e^u
 \end{aligned}
\qquad\text{in }B_2 (0)\backslash \{0\} \subset \mathbb{R}^2
 \end{equation}
 such that
 \begin{equation}\label{equ6.11}
  u(x)\neq O\left( \psi(|x|)h\left(\log\frac{2}{|x|}\right)\right)
\quad \text{as }x\to 0
 \end{equation}
 and
 \begin{equation}\label{equ6.12}
  \frac{v(x)}{\log\frac{1}{|x|}} \to 1 \quad \text{as }x\to0.
 \end{equation}
\end{thm}

\begin{proof}
  Let $v(x)=\log\frac{4}{|x|}$.  Then $v$ satisfies
  \eqref{equ6.10}$_2$ and \eqref{equ6.12}.  Define
  $\varphi:(0,1)\to(0,1)$ by $\varphi=\sqrt{\psi}$.  Let
  $\{x_j\}^{\infty}_{j=1} \subset\mathbb{R}^2$ be as in Lemma
  \ref{lem4.1} and $r_j
  =e^{-\frac{1}{2}h\left(\log\frac{2}{|x_j|}\right)}$.  By taking a
  subsequence if necessary, it follows from \eqref{equ6.9}$_1$ that
  $r_j$ satisfies \eqref{eq4.3}.
 
    Therefore, by Lemma \ref{lem4.1}, there exists a positive function
    $u\in C^\infty (B_2 (0)\backslash \{0\})$ and a positive constant
    $A$ such that $u$ satisfies \eqref{eq4.4}--\eqref{eq4.7}.  Thus
    $u$ satisfies \eqref{equ6.10}$_1$ in $B_2
    (0)\backslash(\{0\}\cup\cup^{\infty}_{j=1} B_{r_j} (x_j))$.  Also
    for $x\in B_{r_j} (x_j)$ we have
 \begin{equation*}
  0\leq -\Delta u(x)\leq \frac{\varphi(|x_j|)}{r^{2}_{j}} \leq \frac{1}{r^{2}_{j}}=e^{h\left(\log\frac{4}{2|x_j|}\right)} <e^{h\left(\log\frac{4}{|x|}\right)} =e^{h(v(x))} .
 \end{equation*}
 Hence $u$ satisfies \eqref{equ6.10}$_1$ in $B_2 (0)\backslash \{0\}$.
 
 Finally
 \begin{align*}
  \frac{u(x_j)}{\psi(|x_j|)h\left(\log\frac{2}{|x_j|}\right)} &\geq \frac{A\varphi(|x_j|)\log\frac{1}{r_j}}{\psi(|x_j|)h\left(\log\frac{2}{|x_j|}\right)}\\
  &=\frac{A/2}{\sqrt{\psi(|x_j|)}}\to\infty \quad\text{ as }j\to\infty
 \end{align*}
 which proves \eqref{equ6.11}.
\end{proof}

\begin{proof}[Proof of Theorem \ref{thm2.5}]
Define functions $u$ and $v$ by 
 \begin{equation*}
  u(x)=U(x)+a\log\frac{1}{|x|}, \qquad v(x)=V(x)+a\log\frac{1}{|x|}.
 \end{equation*}
 Then $u$ and $v$ are $C^2$ positive solutions of
 \begin{align*}
  &0\leq-\Delta u\\
  &0\leq-\Delta v\le e^u
 \end{align*}
 in a punctured neighborhood of the origin.  Thus \eqref{eq2.17}
 follows from Theorem \ref{thm2.3}.  Hence by \eqref{eq2.14}
 \begin{align*}
  \log^{+}(-\Delta u(x))&=\log^{+}(-\Delta U(x))\leq \log^{+}(|x|^{-a} e^{|V|^\lambda} )\\
  &=a\log\frac{1}{|x|}+|V|^\lambda \leq a\log\frac{1}{|x|}+C\left(\log\frac{1}{|x|}\right)^\lambda .
 \end{align*}
 Thus \eqref{eq2.16} follows from Lemma \ref{lem4.3}.
\end{proof}

\section{Proofs of three and higher dimensional results}\label{sec6}

In this section we prove Theorems \ref{thm3.1}--\ref{thm3.7}.

\begin{proof}[Proof of Theorem \ref{thm3.1}]
  Since increasing $\sigma$ and/or $\lambda$ weakens the conditions
  (\ref{eq3.1}, \ref{eq3.2}), we can assume $\sigma=\lambda=\frac{n}{n-2}$.
 
  As in the first paragraph of the proof of Theorem \ref{thm5.1},
  there exist positive constants $K$ and $\varepsilon$ such that $u$
  and $v$ are positive solutions of the system
 \begin{equation*}
 \begin{matrix} 
 0\leq-\Delta u\leq Kv^{\frac{n}{n-2}}\\
 0\leq-\Delta v\leq Ku^{\frac{n}{n-2}}
 \end{matrix}
 \quad\text{in } B_\varepsilon (0)\backslash \{0\}.
 \end{equation*}
 Let $w=u+v$.  Then in $B_\varepsilon (0)\backslash \{0\}$ we have
 \begin{align*}
  0&\leq-\Delta w=-\Delta u-\Delta v\leq K\left(u^{\frac{n}{n-2}} +v^{\frac{n}{n-2}}\right)\\
  &\leq Kw^{\frac{n}{n-2}} .
 \end{align*}
 Thus by \cite[Theorem 2.1]{T2001}
 \begin{equation*}
  u(x)+v(x)=w(x)=O\left(|x|^{-(n-2)}\right) \quad\text{ as }x\to0
 \end{equation*}
 which proves \eqref{eq3.6} and \eqref{eq3.7}.
\end{proof}

\begin{proof}[Proof of Theorem \ref{thm3.5}]
As in the first paragraph of the proof of Theorem \ref{thm5.1}, we can
assume the function $g$ is given by $g(t)=t^\sigma$ and then Theorem
\ref{thm3.5} follows immediately from Lemma \ref{lem4.5}(i) with
$\beta=0$. 
\end{proof}

\begin{proof}[Proof of Theorem \ref{thm3.7}]
We prove Theorem \ref{thm3.7} one case at a time.
 
\begin{description}
  \item[Case A.] Suppose $\sigma=0$.  Then by \eqref{eq3.23} and 
Lemma \ref{lem4.4}(i) applied to $v$ we have
  \begin{equation}\label{eq6.1}
   v(x)=O\left( \left(\frac{1}{|x|}\right)^{n-2}\right) +o\left( \left(\frac{1}{|x|}\right)^{\frac{n-2}{n}\beta} \right).
  \end{equation}
  It follows therefore from \eqref{eq3.22} that
  \begin{equation*}
   -\Delta u(x)=O\left( \left(\frac{1}{|x|}\right)^{(n-2)\lambda+\alpha} +\left(\frac{1}{|x|}\right)^{\frac{n-2}{n}\beta\lambda+\alpha}\right)
  \end{equation*}
  and hence by Lemma \ref{lem4.4}(i)
  \begin{equation}\label{eq6.2}
   u(x)=O\left( \left(\frac{1}{|x|}\right)^{n-2}\right) +o\left( \left(\frac{1}{|x|}\right)^{\frac{n-2}{n}((n-2)\lambda+\alpha)} +\left(\frac{1}{|x|}\right)^{\frac{n-2}{n}\left(\frac{n-2}{n}\beta\lambda+\alpha\right)} \right).
  \end{equation}
  Case A of Theorem \ref{thm3.7} follows immediately from
  \eqref{eq6.1} and \eqref{eq6.2}.
\end{description}  
The reasoning used to prove Cases B, C, and D of Theorem \ref{thm3.7}
is as follows.  Either $u$ satisfies
  \begin{equation}\label{eq6.3}
   -\Delta u(x)=O\left( \left(\frac{1}{|x|}\right)^n \right) 
\quad\text{ as }x\to0
  \end{equation}
  or it doesn't. 

\begin{description}
\item[Step I.] If $u$ satisfies \eqref{eq6.3} then we prove below that
  $u$ and $v$ satisfy \eqref{eq3.25} and \eqref{eq3.26}.
\item[Step II.]  If $u$ does not satisfy \eqref{eq6.3} then, for
  example, to prove Theorem \ref{thm3.7} in Case B, we prove below that the
  condition $\delta\le n$ in (B1) does not hold and $u$ and $v$
  satisfy \eqref{eq3.27} and \eqref{eq3.28}.
\end{description}
These two steps complete the proof of Case B as follows: If the
condition $\delta\le n$ in (B1) holds then by Step II, $u$ satisfies
\eqref{eq6.3} and hence by Step I, $u$ and $v$ satisfy (\ref{eq3.25},
\ref{eq3.26}). On the other hand, if the condition $\delta>n$ in (B2)
holds then by Steps I and II, $u$ and $v$ satisfy either
(\ref{eq3.25}, \ref{eq3.26}) or (\ref{eq3.27}, \ref{eq3.28}). But
since (\ref{eq3.25}, \ref{eq3.26}) implies (\ref{eq3.27},
\ref{eq3.28}), we have $u$ and $v$ satisfy (\ref{eq3.27},
\ref{eq3.28}).

Similar reasoning will be used in Cases C and D.

  \begin{description}
\item[Step I.] Suppose $u$ satisfies \eqref{eq6.3}.
    Then by Lemma \ref{lem4.4}(i) with $\gamma=n$ we see that $u$ satisfies
    \eqref{eq3.25} as $x\to0$.  Hence by \eqref{eq3.23},
  \begin{equation*}
   0\leq-\Delta v=O\left(
     \left(\frac{1}{|x|}\right)^{(n-2)\sigma+\beta}\right) 
\quad\text{ as }x\to0.
  \end{equation*}
  Thus by Lemma \ref{lem4.4}(i) applied to $v$ we have as $x\to0$ that
  \begin{equation*}
   v(x)=O\left( \left(\frac{1}{|x|}\right)^{n-2}\right) +o\left( \left(\frac{1}{|x|}\right)^{\frac{n-2}{n}((n-2)\sigma+\beta)} \right)
  \end{equation*}
  which implies $v$ satisfies \eqref{eq3.26} as $x\to0$. This
  completes the proof of Step I.
  \item[Step II.] Suppose
  \begin{equation}\label{eq6.4}
   -\Delta u(x)\neq O \left( \left(\frac{1}{|x|}\right)^n \right) 
\quad\text{ as }x\to0.
  \end{equation}
  By Lemma \ref{lem4.5}(ii)
  \begin{equation}\label{eq6.5}
   -\Delta u(x)=O\left( \left(\frac{1}{|x|}\right)^{\gamma_1} \right) 
\quad\text{ as }x\to0
  \end{equation}
  for some $\gamma_1 >n$.
 \end{description}
We now complete the proof Theorem \ref{thm3.7} by completing the proof
of Step II one case at a time.
\begin{description}
\item[Case B.]  Suppose $0<\sigma<\frac{2}{n-2}$.  Then by
  Lemma \ref{lem4.5}(i) we have $v(x)$ satisfies
  \eqref{eq3.28}.  Hence by \eqref{eq3.22}
 \begin{equation}\label{eq6.6}
  -\Delta u(x)=O\left( \left(\frac{1}{|x|}\right)^{(n-2)\lambda+\alpha} +\left(\frac{1}{|x|}\right)^{[(n-2)\sigma-2+\beta]\lambda+\alpha} \right).
 \end{equation}
 By \eqref{eq6.4} the maximum $\delta$ of the two exponents on
 $\frac{1}{|x|}$ in \eqref{eq6.6} is greater than $n$.  Thus by Lemma
 \ref{lem4.4}(i), $u$ satisfies \eqref{eq3.27}.  This completes the proof of
   Step II in Case B.
 \item[Case C.] Suppose $\sigma=\frac{2}{n-2}$.  Then by
   \eqref{eq6.5} and Lemma \ref{lem4.4}(iii) we have $v(x)$ satisfies
   \eqref{eq3.30}.  Hence by \eqref{eq3.22}
 \begin{equation*}
  -\Delta u(x)=
  \begin{cases}
   O\left( \left(\frac{1}{|x|}\right)^{(n-2)\lambda+\alpha>n} \right), &\text{ if }\beta<n-2\\
   o\left( \left(\frac{1}{|x|}\right)^{\beta\lambda+\alpha} \left(\log\frac{1}{|x|}\right)^\lambda \right), &\text{ if }\beta\geq n-2.
  \end{cases}
  \end{equation*}
  Thus by \eqref{eq6.4} neither (i) nor (ii) in the statement of Case C holds.
  Hence by Lemma \ref{lem4.4}(i),(ii), $u$ satisfies \eqref{eq3.29}.
  This completes the proof of Step II in Case C.
\item[Case D.] Suppose $\sigma>\frac{2}{n-2}$ and $a$ and
  $b$ are defined by \eqref{eq3.31}.  Then by \eqref{eq3.21}
  \begin{equation}\label{eq6.7}
   1>1-a=\frac{n-2}{n}\lambda \left[ \frac{n}{n-2} \frac{1}{\lambda} +\frac{2}{n-2}-\sigma\right]>0.
  \end{equation}
  By \eqref{eq6.5} and Lemma \ref{lem4.4}(iii) we have
  \begin{equation*}
   v(x)=O\left( \left(\frac{1}{|x|}\right)^{p_0} \right) \quad\text{ as }x\to0
  \end{equation*}
  for some $p_0 >\max \{n-2, \, \frac{b}{1-a}\}$.  Hence by \eqref{eq3.22}
  \begin{equation*}
   -\Delta u(x)=O\left( \left(\frac{1}{|x|}\right)^{\gamma_0 :=\alpha+p_0 \lambda}\right) \quad\text{ as }x\to0
  \end{equation*}
  and $\gamma_0 >n$ by \eqref{eq6.4}.  Thus by Lemma \ref{lem4.4}(iii)
  $v$ satisfies \eqref{eq4.19} with $\gamma=\gamma_0 =\alpha+p_0 \lambda>n$,
  that is
  \begin{equation}\label{eq6.8}
   v(x)=O \left( \left(\frac{1}{|x|}\right)^{n-2} \right) +o\left( \left(\frac{1}{|x|}\right)^{p_1} \right) \quad\text{ as }x\to0
  \end{equation}
  where
  \begin{align*}
   p_1 :&=\frac{\gamma_0}{n}[(n-2)\sigma-2]+\beta=\frac{\alpha+p_0 \lambda}{n} [(n-2)\sigma-2]+\beta\\
   &=p_0 a+b.
  \end{align*}
 By \eqref{eq6.7} the sequence defined by $p_{j+1} =ap_j +b$ decreases
 to $\frac{b}{1-a}$.  Thus after iterating a finite number of times
 the process of obtaining $p_1$ from $p_0$ and using \eqref{eq6.8} we obtain
 as $x\to0$ that $v$ satisfies \eqref{eq3.33} for all $\varepsilon>0$.  Hence
 by \eqref{eq3.22}
 \begin{equation}\label{eq6.9}
  -\Delta u(x)=
  \begin{cases}
   O\left( \left(\frac{1}{|x|}\right)^{(n-2)\lambda+\alpha}\right), &\text{ if }\frac{b}{1-a}<n-2\\
   O\left(
     \left(\frac{1}{|x|}\right)^{\frac{b\lambda}{1-a}+\alpha+\varepsilon}\right),
   &\text{ if }\frac{b}{1-a}\ge n-2   
  \end{cases}
 \end{equation}
 for all $\varepsilon>0$.  By \eqref{eq6.4} the exponents on
 $\frac{1}{|x|}$ in (6.9) are greater than $n$.  (That is neither (i)
 nor (ii) in the statement of Case D hold.)  Thus, by Lemma \ref{lem4.4}(i), $u$
 satisfies \eqref{eq3.32} for all $\varepsilon>0$.  This completes the proof
 of Step II in Case D.
\end{description}
\end{proof}

\begin{proof}[Proof of Theorem \ref{thm3.2}]
  Since increasing $\sigma$ weakens the condition \eqref{eq3.2} on $g$
  and since the bounds \eqref{eq3.9}, \eqref{eq3.10} do not depend on
  $\sigma$, we can assume without loss of generality that
 \begin{equation*}
 \lambda>\frac{n}{n-2} \quad\text{ and }\quad\frac{2}{n-2}<\sigma<\frac{2}{n-2}+\frac{n}{n-2}\frac{1}{\lambda}.
 \end{equation*}
 
 As in the first paragraph of the proof of Theorem \ref{thm5.1}, there exists a
 constant $K>0$ such that $u$ and $v$ are $C^2$ positive solutions of
 \begin{align*}
  &0\leq-\Delta u\leq Kv^\lambda\\
  &0\leq-\Delta v\leq Ku^\sigma\\
 \end{align*}
 in a punctured neighborhood of the origin in $\mathbb{R}^n$.  By
 scaling we can assume $K=1$.
 
 We now apply Theorem \ref{thm3.7}, Case D with $\alpha=\beta=0$.  Let
 $a$ and $b$ be defined by \eqref{eq3.31}.  Then $b=0, \,
 0=\frac{b}{1-a}<n-2$ and $(n-2)\lambda>n=n-\alpha$.  Thus neither (i)
 nor (ii) in Theorem \ref{thm3.7}, Case D, hold.  Hence Theorem
 \ref{thm3.2} follows from \eqref{eq3.32} and \eqref{eq3.33}.
\end{proof}

\begin{proof}[Proof of Theorem \ref{thm3.3}]
  Let $v(x)=|x|^{-(n-2)}$ then $v$ satisfies \eqref{eq3.11}$_2$ and
  \eqref{eq3.14}.  Define $\varphi:(0,1)\to(0,1)$ by
  $\varphi=\sqrt{\psi}$.  Let $\{x_j\}$ be as in Lemma \ref{lem4.1}
  and $r_j =(2|x_j|)^{\frac{n-2}{n}\lambda}$.  By taking a
  subsequence if necessary, $r_j$ satisfies \eqref{eq4.3}.
  Therefore by Lemma \ref{lem4.1} there exists a positive function
  $u\in C^\infty (\mathbb{R}^n \backslash \{0\})$ and a positive
  constant $A=A(n)$ such that $u$ satisfies
  \eqref{eq4.4}--\eqref{eq4.7}.  Thus $u$ satisfies \eqref{eq3.11}$_1$ in
  $\mathbb{R}^n \backslash(\{0\}\cup\cup^{\infty}_{j=1} B_{r_j}
  (x_j))$.  Also for $x\in B_{r_j} (x_j)$ we have
 \begin{equation*}
  0\leq-\Delta u(x)\leq \frac{\varphi(|x_j|)}{r^{n}_j}<\left(\frac{1}{2|x_j|}\right)^{(n-2)\lambda} <v(x)^\lambda.
 \end{equation*}
 Hence $u$ satisfies \eqref{eq3.11}$_1$ in $\mathbb{R}^n \backslash \{0\}$.

 Finally,
 \begin{align*}
  \frac{u(x_j)}{\psi(|x_j|)|x_j|^{-\frac{(n-2)^2}{n}\lambda}}&\geq \frac{A\varphi(|x_j|)}{r^{n-2}_{j} \psi(|x_j|)|x_j|^{-\frac{(n-2)^2}{n}\lambda}}\\
  &=\frac{A}{2^{\frac{(n-2)^2}{n}\lambda} \sqrt{\psi(|x_j|)}}\to\infty \quad\text{ as }j\to\infty
 \end{align*}
 which proves \eqref{eq3.13}.
\end{proof}

\begin{proof}[Proof of Theorem \ref{thm3.4}]
  It follows from \eqref{eq3.15} that $\lambda>\frac{n}{n-2}$.  Denote the
  problem \eqref{eq3.16} by $P(\lambda,\sigma)$.  If 
$\hat{\lambda}\geq\lambda$ and $\hat\sigma\geq\sigma$
  are constants and $(u,v)$ solves $P(\lambda,\sigma)$ then clearly
  $(u,v)$ solves $P(\hat{\lambda}, \hat{\sigma})$.  We can therefore
  assume
 \begin{equation}\label{eq6.10}
 \sigma<\frac{n}{n-2}.
 \end{equation}
 Since the first inequality in \eqref{eq3.15} holds
 \begin{align*}
  &\text{if and only if} \qquad (n-2)\sigma>2+\frac{n}{\lambda}=n-(n-2)+\frac{n}{\lambda}=n-\frac{(n-2)\lambda-n}{\lambda}\\
  &\text{if and only if} \qquad n-(n-2)\sigma<\frac{(n-2)\lambda-n}{\lambda}
 \end{align*}
 we see by \eqref{eq6.10} that
 \begin{equation*}
  \frac{1}{n-(n-2)\sigma}>\frac{\lambda}{(n-2)\lambda-n},
 \end{equation*}
 or, in other words,
 \begin{equation}\label{eq6.11}
  \beta>\alpha\lambda>0 \quad\text{ where }\quad
  \beta:=\frac{1}{n-(n-2)\sigma} \quad\text{ and }\quad \alpha:=\frac{1}{(n-2)\lambda-n}.
 \end{equation}
 Let $\varphi(r)=r$ and let $\{x_j\}^{\infty}_{j=1}, \,
 \{r_j\}^{\infty}_{j=1}$, and $A=A(n)$ be as in Lemma \ref{lem4.1}.
 Define $\psi_j >0$ as a function of $r_j$ by
 \begin{equation}\label{eq6.12}
  r_j =\left(\frac{(A\psi_j)^\lambda}{\varphi(|x_j|)}\right)^\alpha .
 \end{equation}
 Then
 \begin{equation}\label{eq6.13}
  \frac{A\psi_j}{r^{n-2}_{j}}=\frac{(A\psi_j)\varphi(|x_j|)^{\alpha(n-2)}}{(A\psi_j)^{\lambda\alpha(n-2)}}=\left(\frac{\varphi(|x_j|)^{n-2}}{(A\psi_j)^{\lambda(n-2)-\frac{1}{\alpha}=n}}\right)^\alpha .
 \end{equation}
 By decreasing $r_j$ (and thereby decreasing $\psi_j$) we can assume
 \begin{equation}\label{eq6.14}
  \frac{A\varphi(|x_j|)}{r^{n-2}_{j}} >h(|x_j|)^2, \qquad \sum^{\infty}_{j=1} \psi_j <\infty,
 \end{equation}
 \begin{equation}\label{eq6.15}
  \psi^{\alpha\lambda-\beta}_{j} \geq
  \frac{\varphi(|x_j|)^{\alpha-\sigma\beta}}{A^{\sigma\beta+\alpha\lambda}}
  \qquad\text{ and }\qquad \frac{A\psi_j}{r^{n-2}_{j}} >h(|x_j|)^2
 \end{equation}
 by \eqref{eq6.13}. It follows from \eqref{eq6.12} and \eqref{eq6.15}
 that
 \begin{equation*}
  \left( \frac{(A\psi_j)^\lambda}{\varphi(|x_j|)}\right)^\alpha =r_j \geq \left(\frac{\psi_j}{(A\varphi(|x_j|))^\sigma}\right)^\beta
 \end{equation*}
 which implies
 \begin{equation}\label{eq6.16}
  0<\frac{\varphi(|x_j|)}{r^{n}_{j}}=\left(\frac{A\psi_j}{r^{n-2}_{j}}\right)^\lambda
 \end{equation}
 and
 \begin{equation}\label{eq6.17}
  0<\frac{\psi_j}{r^{n}_{j}} \leq \left(\frac{A\varphi(|x_j|)}{r^{n-2}_{j}}\right)^\sigma .
 \end{equation}
 Let $\psi:(0,1)\to(0,1)$ be a continuous function such that
 $\psi(|x_j|)=\psi_j$.  By Lemma \ref{lem4.1} there exist positive functions
 $u,v\in C^\infty (\mathbb{R}^n \backslash \{0\})$ such that $u$
 satisfies \eqref{eq4.4}--\eqref{eq4.7} and $v$ satisfies
 \begin{equation}\label{eq6.18}
  0\leq-\Delta v\leq\frac{\psi(|x_j|)}{r^{n}_{j}} \quad\text{ in }B_{r_j} (x_j)
 \end{equation}
 \begin{equation}\label{eq6.19}
  -\Delta v=0 \quad\text{ in }\mathbb{R}^n \backslash \left(\{0\}\cup\bigcup^{\infty}_{j=1} B_{r_j} (x_j)\right)
 \end{equation}
 \begin{equation}\label{eq6.20}
  v\geq\frac{A\psi(|x_j|)}{r^{n-2}_{j}} \quad\text{ in }B_{r_j} (x_j)
 \end{equation}
 and
 \begin{equation}\label{eq6.21}
  v\geq 1\quad\text{ in }\mathbb{R}^n \backslash \{0\}.
 \end{equation}
 Theorem \ref{thm3.4} follows from \eqref{eq4.4}--\eqref{eq4.7},
 \eqref{eq6.18}--\eqref{eq6.21}, \eqref{eq6.16}, \eqref{eq6.17},
 \eqref{eq6.14}$_1$, and \eqref{eq6.15}$_2$.
\end{proof}

\begin{proof}[Proof of Theorem \ref{thm3.6}]
  Let $u(x)$ and $v(x)$ be the Kelvin transforms of $U(y)$ and $V(y)$
  respectively.  Then
 \begin{equation}\label{eq6.22}
  U(y)=|x|^{n-2} u(x), \qquad V(y)=|x|^{n-2} v(x), \qquad x=\frac{y}{|y|^2}
 \end{equation}
 \begin{align*}
  &\Delta U=|x|^{n+2} \Delta u, \qquad \Delta V=|x|^{n+2} \Delta v\\
  &U+1=|x|^{n-2} (u+|x|^{-(n-2)} ), \qquad V+1=|x|^{n-2} (v+|x|^{-(n-2)} )
 \end{align*}
 and thus $u(x)$ and $v(x)$ are $C^2$ nonnegative solutions of the
 system (\ref{eq3.22}, \ref{eq3.23}) in a punctured neighborhood of the
 origin where
 \begin{equation}\label{eq6.23}
  \alpha=n+2-(n-2)\lambda \quad\text{ and }\quad\beta=n+2-(n-2)\sigma.
 \end{equation}
 Using Theorem \ref{thm3.7} we get the following results.
 \begin{description}
  \item[Case A.] Suppose $\sigma=0$.  Then $\beta=n+2$, 
\[
\frac{n-2}{n}\beta=\frac{(n-2)(n+2)}{n}, 
\]
and
  \begin{align*}
   \frac{n-2}{n}&\left(\frac{n-2}{n}\beta\lambda+\alpha\right) =\frac{n-2}{n}\left(\frac{(n-2)(n+2)}{n}\lambda +n+2-(n-2)\lambda\right)\\
   &=\frac{n-2}{n}\left[ \left(\frac{n+2}{n}-1\right)(n-2)\lambda+n+2\right]\\
   &=\frac{n-2}{n}\left[\frac{2(n-2)}{n}\lambda+n+2\right]\geq(n-2)\frac{n+2}{n}>n-2.
  \end{align*}
  Thus by Theorem \ref{thm3.7}(A2) we have
  \begin{align*}
   &u(x)=o\left( \left(\frac{1}{|x|}\right)^{\frac{n-2}{n}\left[\frac{2(n-2)}{n}\lambda+n+2\right]}\right) \quad\text{ as }x\to0\\
   &v(x)=o\left(
     \left(\frac{1}{|x|}\right)^{(n-2)\left(1+\frac{2}{n}\right)}
   \right) 
\quad\text{ as }x\to0.
  \end{align*}
  Hence Case A of Theorem \ref{thm3.6} follows from \eqref{eq6.22}.
  \item[Case B.] Suppose $0<\sigma<\frac{2}{n-2}$.  Then
  \begin{align*}
   &(n-2)\lambda+\alpha=n+2, \quad (n-2)\sigma+\beta=n+2\\
   &\lambda[(n-2)\sigma-2+\beta]+\alpha=\lambda n+n+2-(n-2)\lambda=n+2+2\lambda
  \end{align*}
   and
   \begin{equation*}
    \delta=\max\{n+2, \, n+2+2\lambda\}=n+2+2\lambda>n.
   \end{equation*}
   Thus by Theorem \ref{thm3.7}(B2) we have
   \begin{align*}
    &u(x)=o\left( \left(\frac{1}{|x|}\right)^{(n-2)\frac{n+2+2\lambda}{n}}\right) \quad\text{ as }x\to0\\
    &v(x)=O\left( \left(\frac{1}{|x|}\right)^n \right) \quad\text{ as }x\to0.
   \end{align*}
   Hence Case B of Theorem \ref{thm3.6} follows from \eqref{eq6.22}.
   \item[Case C.] Suppose $\sigma=\frac{2}{n-2}$.  Then
   \begin{align*}
    &\beta=n+2-(n-2)\sigma=n>n-2\\
    &\alpha=n+2-(n-2)\lambda
   \end{align*}
   and
   \begin{equation*}
    \beta\lambda+\alpha=n\lambda+n+2-(n-2)\lambda=n+2(\lambda+1)>n+2.
   \end{equation*}
   Thus by Theorem \ref{thm3.7}(C2) we have
   \begin{align*}
    &u(x)=o\left( \left(\frac{1}{|x|}\right)^{(n-2)\left(1+\frac{2(\lambda+1)}{n}\right)} \left(\log\frac{1}{|x|}\right)^{\frac{n-2}{n}\lambda} \right) \quad\text{ as }x\to0\\
    &v(x)=o\left( \left(\frac{1}{|x|}\right)^n
      \log\frac{1}{|x|}\right) \quad \text{ as }x\to0.
   \end{align*}
   Hence Case C of Theorem \ref{thm3.6} follows from \eqref{eq6.22}.
 \item[Case D.] Suppose $\sigma>\frac{2}{n-2}$.  Let $a$
   and $b$ be defined by \eqref{eq3.31}.  Then by \eqref{eq6.23}
and direct calculation (Maple is helpful), we find   

\begin{equation*}
    \frac{b}{1-a}=(n-2)\left[1+\frac{2\sigma+2}{D}\right]>n-2
   \end{equation*}
   and
   \begin{equation}\label{eq6.24} 
    \frac{b\lambda}{1-a}-(n-\alpha)=\frac{2n(\lambda+1)}{D}>0
   \end{equation}
   by \eqref{eq3.21}.  Thus neither (i) nor (ii) in Theorem
   \ref{thm3.7}(D1) hold.  Also \eqref{eq6.24} implies
   \begin{equation*}
    \frac{b\lambda}{1-a}+\alpha=n\left[1+\frac{2(\lambda+1)}{D}\right].
   \end{equation*}
   Hence by Theorem \ref{thm3.7}(D2) we have
   \begin{align*}
    &u(x)=o\left( \left(\frac{1}{|x|}\right)^{(n-2)\left(1+\frac{2(\lambda+1)}{D}+\varepsilon\right)} \right) \quad\text{ as }x\to0\\
    &v(x)=o\left( \left(\frac{1}{|x|}\right)^{(n-2)\left(1+\frac{2(\sigma+1)}{D}\right)+\varepsilon}\right) \quad\text{ as }x\to0.
   \end{align*}
   Thus Case D of Theorem \ref{thm3.6} follows from \eqref{eq6.22}.
 \end{description}
\end{proof}

\appendix
\section{Brezis-Lions result}\label{secA}

We use repeatedly the following special case of a result of Brezis 
and Lions \cite{BL}.

\begin{lem}\label{A.1}
  Suppose $u$ is a $C^2$ nonnegative superharmonic function in
  $B_{2\varepsilon} (0)\backslash \{0\}\subset \mathbb{R}^n , \, n\geq
  2$, for some $\varepsilon>0$.  Then
 \begin{equation*}
  \int_{|x|<\varepsilon} -\Delta u(x)\,dx<\infty
 \end{equation*}
 and for $0<|x|<\varepsilon$ we have
 \begin{equation*}
  u(x)=m\Gamma(|x|)+\int_{|y|<\varepsilon} \omega\Gamma(|x-y|)(-\Delta u(y))\,dy+h(x)
 \end{equation*}
 where $\Gamma$ is given by \eqref{eq1.5}, $\omega=\omega(n)>0$ and
 $m\ge 0$ are constants, $\omega(2)=\frac{1}{2\pi}$, and
 $h:B_\varepsilon (0)\to\mathbb{R}$ is harmonic.
\end{lem}


\begin{thebibliography}{33}

\bibitem{AH} D.~R. Adams and L.~I. Hedberg, Function Spaces and
  Potential Theory, Grundlehren der Math. Wissenschaften 314,
  Springer, Berlin-Heidelberg, 1996.

\bibitem{BVG1999} M.~F. Bidaut-V\'eron and P.~Grillot, Singularities in
  elliptic systems with absorbtion terms, Ann. Scuola Norm. Sup. Pisa
  Cl. Sci. 28 (1999), 229--271.


\bibitem{BVP2001} M.~F. Bidaut-V\'eron and S.~Pohozaev, Nonexistence
  results and estimates for some nonlinear elliptic problems,
  J. Analyse Math. 84 (2001), 1--49.


\bibitem{BL} H.~Brezis and P-L.~Lions, A note on isolated
  singularities for linear elliptic equations, Mathematical analysis
  and applications, Part A, pp. 263–-266, Adv. in Math. Suppl. Stud.,
  7a, Academic Press, New York-London, 1981.

\bibitem{BM} H.~Brezis and F.~Merle, Uniform estimates and blow-up
  behavior for solutions of $-\Delta u=V(x)e^u$ in two dimensions,
  Comm. Partial Differential Equations 16 (1991), 1223–-1253.

\bibitem{GT} D. Gilbarg and N. Trudinger, Elliptic partial
  differential equations of second order, Second edition,
  Springer-Verlag, Berlin, 1983.

\bibitem{KM} T.~Kilpel\"{a}inen and J.~Mal\'y, The {W}iener test and
  potential estimates for quasilinear elliptic equations, Acta Math.
  172 (1994), 137--161.

\bibitem{Lab} D.~Labutin, Potential estimates for a class of fully
  nonlinear elliptic equations, Duke Math. J. 111 (2002), 1--49.

\bibitem{Maz} V.~Maz'ya, Sobolev Spaces, with Applications to Elliptic
  Partial Differential Equations, 2nd, Augmented Edition. Grundlehren
  der Math. Wissenschaften 342, Springer, Berlin, 2011.

\bibitem{M1993} E.~Mitidieri, A Rellich type identity and
  applications, Comm. Partial Differential Equations 18 (1993),
  125--151.

\bibitem{M1996} E.~Mitidieri, Nonexistence of positive solutions of
  semilinear elliptic systems in $\mathbb{R}^N$, Differential Integral
  Equations 9 (1996), 465--479.

\bibitem{MP2001} E.~Mitidieri and S.~Pohozaev, A priori estimates and the
  absence of solutions of nonlinear partial differential equations and
  inequalities. (Russian) Tr. Mat. Inst. Steklova 234 (2001), 1--384;
  translation in Proc. Steklov Inst. Math. 2001, no. 3 (234), 1–-362

\bibitem{PQS2007} P.~Pol\'a\v cik, P.~Quittner, and Ph.~Souplet,
  Singularity and decay estimates in superlinear problems via
  Liouville-type theorems, Part I: Elliptic systems, Duke Math. J. 139
  (2007), 555--579.

\bibitem{PV}N. C.~Phuc and I.~E.~Verbitsky, Quasilinear and Hessian
  equations of Lane--Emden type, Ann. Math. 168 (2008), 859--914.

\bibitem{S2009} Ph.~Souplet, The proof of the Lane-Emden conjecture in
  four space dimensions, Adv. Math. 221 (2009), 1409--1427.

\bibitem{T2001} S.~D.~Taliaferro, Isolated singularities of nonlinear
  elliptic inequalities, Indiana Univ. Math. J. 50 (2001), 1885--1897.

\bibitem{T2006} S.~D.~Taliaferro, Isolated singularities of nonlinear
  elliptic inequalities. II. Asymptotic behavior of solutions, Indiana
  Univ. Math. J. 55 (2006), 1791--1812.

\bibitem{T2013} S.~D.~Taliaferro, Pointwise bounds and blow-up for
  nonlinear polyharmonic inequalities, Ann. Inst. H. Poincar\'e
  Anal. Non Lin\'eaire 30 (2013), 1069--1096.

\bibitem{Ver} I.~E.~Verbitsky, Nonlinear potentials and trace
  inequalities, Oper. Theory Adv. Appl. 110 (1999), 323--343.


\end{thebibliography}
\end{document}